\documentclass[a4paper,10pt]{article}
\usepackage[a4paper, total={6in, 8in}]{geometry}
\usepackage[utf8]{inputenc}
\usepackage[english]{babel}
\usepackage[T1]{fontenc}
\usepackage[export]{adjustbox}
\usepackage[font=small,labelfont=bf,tableposition=top]{caption}
\DeclareCaptionLabelFormat{andtable}{#1~#2  \&  \tablename~\thetable}
\usepackage{graphicx} 
\usepackage{amsmath,amsfonts,amssymb,amsthm}
\usepackage{mathtools}
\mathtoolsset{showonlyrefs=true}
\usepackage{amsfonts}
\usepackage[ruled,vlined]{algorithm2e}
\usepackage{lipsum}
\usepackage[normalem]{ulem}
\usepackage{lmodern}
\usepackage{tcolorbox}
\usepackage{mathabx}
\usepackage{caption} 
\usepackage{csquotes}
\usepackage{bbm}

\newtheorem{prop}{Proposition}

\newtheorem{lemma}{Lemma}
\newtheorem{cor}{Corollary}
\newtheorem{definition}{Definition}

\newcommand{\Lrm}{\mathrm{L}}
\newcommand{\z}{\boldsymbol{\mathrm{z}}}

\newcommand{\B}{\Psi}

\DeclareMathOperator*{\argmin}{arg\,min}
\newcommand{\mathleft}{\@fleqntrue\@mathmargin0pt}
\newcommand{\x}{\boldsymbol{\tau}}
\newcommand{\db}{\boldsymbol{d}}
\newcommand{\1}{\boldsymbol{1}}
\newcommand{\w}{\boldsymbol{w}}
\newcommand{\0}{\boldsymbol{0}}
\newcommand{\zpulse}{\boldsymbol{\z^{1}}}
\newcommand{\zchase}{\boldsymbol{\z^{0}}}
\newcommand{\zd}{\boldsymbol{\z^{d}}}

\usepackage[isbn=false,eprint=false, defernumbers=true,sorting = none, giveninits, date=year, backend=bibtex]{biblatex}

\begin{document}
\begin{center}
\textbf{\Large Identifying a piecewise affine signal from its nonlinear observation - application to DNA replication analysis}\\[\baselineskip]

{Clara Lage$^1$, Nelly Pustelnik$^1$, Jean-Michel Arbona$^2$, Benjamin Audit$^1$, Rémi Gribonval$^3$}\\[\baselineskip]

{\small
$^1$ Univ Lyon, ENS de Lyon, CNRS, Laboratoire de Physique, F-69342 Lyon, France;
$^2$ Laboratoire de Biologie et Modélisation de la Cellule, ENS de Lyon, Lyon, France;
$^3$ Univ Lyon, ENS de Lyon, Inria, CNRS, UCBL, LIP UMR 5668, F-69342 Lyon, France.}
\end{center}

\begin{abstract}
DNA replication stands as one of the fundamental biological processes crucial for cellular functioning. Recent experimental developments enable the study of replication dynamics at the single-molecule level for complete genomes, facilitating a deeper understanding of its main parameters. In these new data, replication dynamics is reported by the incorporation of an exogenous chemical, whose intra-cellular concentration follows a nonlinear function. The analysis of replication traces thus gives rise to a nonlinear inverse problem, presenting a nonconvex optimization challenge. We demonstrate that under noiseless conditions, the replication dynamics can be uniquely identified by the proposed model. Computing a global solution to this optimization problem is specially challenging because of its multiple local minima.  We present the DNA-inverse optimization method that is capable of finding this global solution even in the presence of noise. Comparative analysis against state-of-the-art optimization methods highlights the superior computational efficiency of our approach. DNA-inverse enables the automatic recovery of all configurations of the replication dynamics, which was not possible with previous methods.

\end{abstract}

\section{Introduction.}
\paragraph{Context} DNA replication is the cellular process by which a cell makes an identical copy of all its chromosomes. It is a highly parallelized DNA synthesis process under strong biological regulation. Its successful completion at each cell cycle is crucial to ensure that genetic information is accurately passed on from one generation to the next. Genetic diseases can appear from replication errors in the germline while genetic instabilities associated to perturbations of the replication dynamic is a recurrent pattern in the appearance and progression of cancer \cite{Gaillard2015}. Hence, the characterization of the so-called \emph{DNA replication program} is not only of fundamental interest but also as implication on human health. 

The replication program for one cell can be described by the replication time versus chromosome position curve: $\tau(x)$ (Figure~\ref{fig:linear_piecewise_linear_rep}A). Single-molecule experimental characterization techniques (i) submit the cells to a pulse of a modified nucleotide called BrdU so that the intracellular BrdU concentration follows a time pulse $\psi(t)$ (Figure~\ref{fig:linear_piecewise_linear_rep}C) and (ii) measure a posteriori the resulting BrdU incorporation profile $z(x)$ along single DNA molecules (Figure~\ref{fig:linear_piecewise_linear_rep}B). The task of characterizing the DNA replication program thus consists in inferring $\tau(x)$ from $z(x)$ given $\psi(t)$ by solving the inverse problem $z(x)=\psi(\tau(x))$. Following previous work in the field, we assume that the rate of DNA synthesis is locally constant, i.e., that $\tau(x)$ is piecewise linear. This configurations results in a non-linear inverse problem defined over the set of function with sparse second derivative.

Nonlinear inverse problems are a field in expansion and with significant applications in imagery, optics, and tomography. \cite{Tomography, phase_pratical_2,phase_GESPAR,Valkonen_accelerationprimaldual,ML_nonlinear_measurement}. In contrast to linear inverse problems, there is no optimization method capable of providing global solutions to large classes of problems \cite{Chambolle,Kunze_inverse_problem}. On the other hand, recent works have successfully adapted methods used in linear inverse problems, such as proximal methods and ADMM, to find local solutions in a nonlinear framework \cite{ADMM_convergence_2,Valkonen_accelerationprimaldual}. The difficulty of transitioning from a local solution to a global one stems from the fact that a nonlinear measurement operator, as in the case of DNA replication, often corresponds to a nonconvex optimization problem. 
In this work, we propose to tackle a nonlinear inverse problem by an approximation that yields a mixed-integer nonlinear programming problem (MINLP). The approach of solving nonconvex optimization problems by introducing integer variables is present in recent literature, particularly in the context of nonconvex machine learning problems \cite{ML_Global,ML_Global_preprint} and to replace non-convex constraints \cite{Integer_constraints}. In our case study, the integer variable enables us to decompose a nonconvex problem into a family of convex problems.

\paragraph{State-of-the-art} 
The \emph{replication timing profile} $\tau(x)$ depends on the location and time of activation of the so-called \emph{replication origins} and the speed of the \emph{replication forks} (Figure~\ref{fig:linear_piecewise_linear_rep}A).
The experimental signals obtained by FORK-seq \cite{Forkseq,biology} captures the DNA synthesis by measuring the variation of the concentration of BrdU, a modified nucleotide that incorporates in replacement of thymidines along a fragment of chromosome, called a \emph{read} (Figure~\ref{fig:linear_piecewise_linear_rep}). Typical examples of experimental signals are illustrated in Figure~\ref{fig:reads}. The signal resulting from one fork is an atom having the shape of the BrdU time pulse $\psi(t)$ with a spatial dilatation depending on the local replication velocity (Figure~\ref{fig:reads}F). 
In previous studies \cite{biology, Forkseq,  LageGretsi, LageHybrid}, signal processing methods applied to FORK-seq data enabled to estimate the position, orientation, and speed of DNA replication forks, but only in replication fork configurations where fork atoms are sufficiently isolated (Figure~\ref{fig:reads}A,D). 
In \cite{biology}, the numerical approach involves a piecewise linear approximation of the BrdU vs.\ space signal. In \cite{LageGretsi,LageHybrid} the function $\psi$ is used as a reference atom in a dictionary composed by translation and rescaling of $\psi$ (Figure~\ref{fig:reads}F) leading to a sparse coding approach. According to \cite{LageGretsi}, the most effective numerical method for sparse coding in signals with high noise is Matching Pursuit \cite{MP}. 

While successful in accurately estimating fork speed, these approaches fail to characterize replication motifs involving truncated atoms that appears in the vicinity of replication origins and termini (Figure~\ref{fig:linear_piecewise_linear_rep}B).
Indeed, since fork progression start at origins and ends when converging forks merge (each loci is replicated once and only once), the significant particularity of the BrdU incorporation signals is that the contributions of distinct forks never overlap and add up, fork atoms being truncated at replication origins and termini (Figure~\ref{fig:linear_piecewise_linear_rep}B).
Therefore, there is a need for an alternative approach capable of robustly extracting these diverging or converging fork configurations (see Figure~\ref{fig:reads}B,E). In this context, we move away from the additive atom approach and instead focus on a method that can determine the time profile directly by specifying a nonlinear inverse problem.



\paragraph{Contribution and outline } The contribution of this article is twofold. We introduce, for the first time, a nonlinear inverse problem that accurately models the biological configuration of FORK-seq data. We thoroughly investigate the theoretical aspects of the model and demonstrate that the optimization problem yields a unique solution under specific assumptions. On the other side, we propose an original numerical method capable of globally solving this problem, a task that is recent investigated  in the existing literature \cite{phase_retrieval_optics,phase_GESPAR,ML_nonlinear_measurement}.

The article is divided as follows: In Section \ref{sec:model} we introduce the DNA-inverse model that is able to capture the main aspects of the biological configuration that results in FORK-seq data as well as the nonlinear inverse problem that provides the timing profile associated to a signal. The theoretical analysis of this problem is provided in Section \ref{sec:optm_guarantees}, where we study conditions for a unique solution. In Section \ref{sec:dna_inverse}, we reformulate the optimization problem to propose a numerical method capable of finding a global solution. The development of this numerical method and algorithm is presented in Section \ref{sec:numericalmethod}. Finally, Section \ref{sec:numresults} presents the numerical results and compares it with results obtained using a state-of-the-art method.

\section{DNA-Inverse Model}\label{sec:model}

\begin{figure}[t]
\centering
\includegraphics[width = 0.95\linewidth]{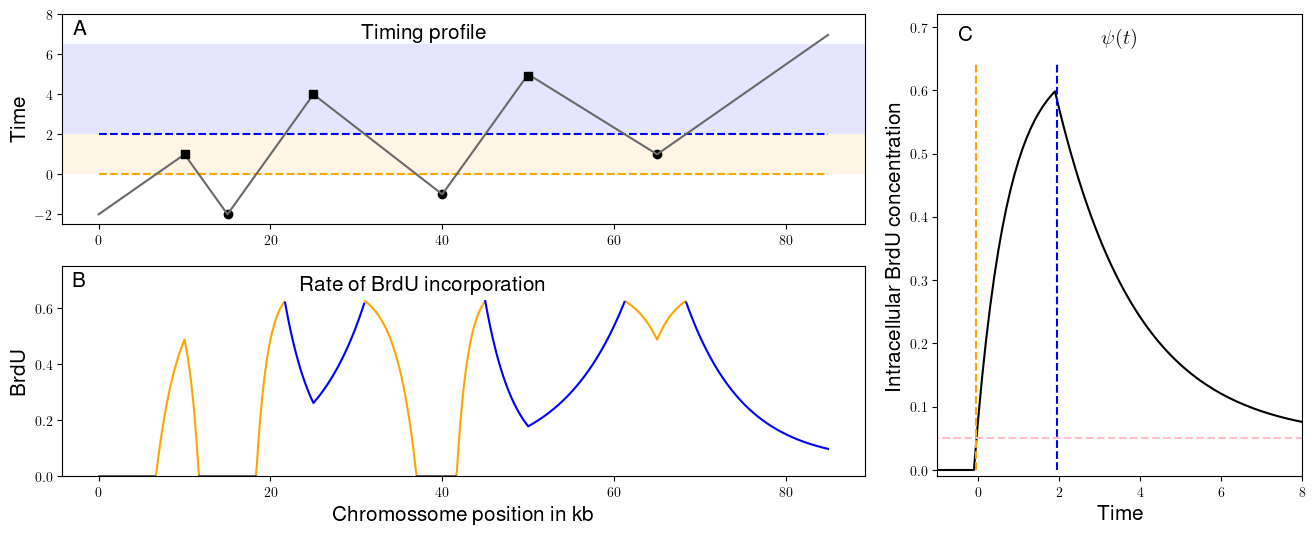}
\caption{The DNA replication program and its characterization by pulse labeling.
(A) Replication timing profile $\tau(x)$ capturing the time of replication during one replication cycle of each loci along a chromosome. DNA replication initiates at multiple sites (black dots), called \emph{replication origins}; from each site two diverging \emph{replication forks} emerge ensuring sequential DNA synthesis; replication terminates by the mergers of converging forks originating from neighboring origins at \emph{replication termini} (black squares). In this example, each replication fork have a constant speed so that the timing profile between origins and termini is linear.
(B) Profile $z(x)$ of the rate of BrdU incorporation along the newly synthesized DNA molecule during the replication program presented in (A) in the presence of the time pulse of BrdU $\psi(t)$ shown in (C); in the absence of noise, it is a simple composition $\psi$ and $\tau$: $z(x) = \psi(\tau(x))$.  
(C) Intracellular BrdU concentration along time, $\psi:\mathbb{R} \rightarrow \mathbb{R}_{+}$; the orange dashed line indicates the beginning of the external BrdU pulse resulting in the progressive increase of the intracellular BrdU concentration, the \emph{pulse phase}; the blue dashed line indicates the end time of the external pulse (dilution of the external media) resulting in the decrease of the internal BrdU concentration, the \emph{chase phase}. The pink dashed line represents the limit of $\psi$ as $t$ approaches infinity.
The background color in (A) and the curve color in (B) correspond to the pulse (orange) and chase (blue) phase, respectively.
}\label{fig:linear_piecewise_linear_rep}
\end{figure}
\paragraph{Model:} Let $\mathcal{X}$ be the discrete set of positions in a certain DNA fragment of size $n$. Then
$\mathcal{X} := \{x_{1},...,x_{n}\} \subset \mathbb{R},$  where $x_{i+1} = x_{i} + \Delta x,$ for $i \in \{1,...,n-1\},$ and $\Delta x > 0$ is fixed as a certain distance in the chromosome scale. Assuming a locally constant speed, the motion of the molecular motor can be characterized by a piecewise linear function $\tau: \mathbb{R} \rightarrow \mathbb{R},$ that assigns to each position of the DNA fragment, the time, starting from the beginning of the experiment, at which this position have been replicated. Negative values of $\tau(x_{i})$ mean that the correspondent region of the DNA fragment have been replicated before the beginning of the experiment.  When restricted to positions in $\mathcal{X},$ the function $\mathbf{\boldsymbol{\tau}}$ is associated with the vector $\boldsymbol{\tau} = (\tau_{1}, \ldots, \tau_{n}),$ where $\tau_{i} := \tau(x_{i}),$ for $i \in \{1,...,n\}$. The vector $\boldsymbol{\tau}$ is referred as a \textit{timing profile} and is illustrated in Figure \ref{fig:linear_piecewise_linear_rep}~A.

The measurement $\Psi$ is defined as the coordinatewise composition of a nonlinear function $\psi$,  which measures the concentration of BrdU in time, and the \textit{timing profile}:
\begin{align*}
  \Psi: &~\mathbb{R}^{n} ~\hspace{9mm} \longrightarrow \hspace{8mm} \mathbb{R}_{+}^{n}\\
& ~ \boldsymbol{\tau} = (\tau_{1},...,\tau_{n}) \mapsto \Psi(\boldsymbol{\tau}) = (\psi(\tau_{1}),...,\psi(\tau_{n})).
\end{align*}
Denoting $\z$ a signal provided by FORK-seq, we suppose that there exists a timing profile $\boldsymbol{\widebar{\tau}} \in \mathbb{R}^{n}$ such that:
$$\z \approx \Psi(\boldsymbol{\widebar{\tau}}),$$
and that $\boldsymbol{\widebar{\tau}}$ is in the set of piecewise linear vectors with a maximum of $C$ breakpoints. This set can be expressed using the $\ell_{0}$ pseudo-norm and a linear operator $\mathrm{L}$ that represents a discrete second derivative: $\mathrm{L}\x = \ell \ast \x,~ \text{with}~ \ell = [1,-2,1],$ where $\mathrm{L}:\mathbb{R}^{n} \rightarrow \mathbb{R}^{n-2}$ does not consider derivatives from the borders. We denote:
\begin{equation}\label{eq:PC}
\mathcal{P}_{C} := \{\x~:~\left\| \mathrm{L}\x \right\|_{0} \leq C \}
\end{equation}   
for some fixed $C \in \mathbb{R}_{+}.$ The non-negativity constraint reflects the fact that the operator $\Psi$ returns zero for negative components of $\x$.  In these conditions, a natural way to estimate $\boldsymbol{\bar{\tau}}$ is to solve the following optimization problem:
\begin{align*}\label{eq:Problem1}
\tag{\textbf{P1}} \widehat{\x} := \argmin_{\x \in \mathcal{P}_{C}} \|\z - \Psi(\x) \|_{2}^{2}, 
\end{align*}
\paragraph{Nonlinear sparse coding problem:} Problem \eqref{eq:Problem1} is a \textit{nonlinear sparse coding problem}. Nonlinear sparse coding problems appear in different application contexts such as partial differential equations, quantization and problems with large application field such as phase-retrieval \cite{Valkonen2021, 1bit-mes, sparseRec,phase_pratical_2,phase_retrieval_optics}. When $\Psi$ is a linear transform,  problem \eqref{eq:Problem1} can be relaxed and solved approximately by $\ell_{1}$ regularization. The resulting optimization problem is known as the \textit{generalized lasso} \cite{generalizedLASSO, Dual_genlasso}. In the general case, the $\ell_{1}$ regularized version of \eqref{eq:Problem1} fits the primal dual formulation for non-convex optimization and can be solved by a primal-dual proximal method or a generalized  Alternating Direction Method of Multipliers (ADMM)  for nonlinear operators $\Psi$ \cite{Valkonen2021,ADMM_convergence_nonconvex}. However, the solution proposed by these methods, as a solution to a non-convex problem, is a local solution. Additionally, these algorithms may have a high runtime depending on the difficulty in calculating the necessary proximal operators for the iterations.

\paragraph{Particularities of DNA-Inverse:} In our case study, we present some specificities of the non-linearity of operator $\Psi$. We note that the structure of the function $\Psi$ is given coordinate-wise by the function $\psi,$ that represents the BrdU concentration. Despite not being invertible, $\psi$ has at most two possible antecedents (``inverses'') for each point $b \in \psi(\mathbb{R}_{+}).$ In addition, any element of $\mathbb{R}_{-}$ is mapped to $0.$
\begin{tcolorbox}[colback=blue!5!white,colframe=blue!75!black]
\begin{equation*}\label{eq:h}
\begin{aligned}
\textbf{A.1}:~ \# \psi^{-1}(b) = \#\{ \boldsymbol{\tau} : \psi(\boldsymbol{\tau}) = b\} \leq 2,~~ \forall b \in \mathbb{R}_{+} \text { and } \psi|_{[-\infty, 0]} = 0. 
\end{aligned}
\end{equation*}
\end{tcolorbox}

The hypothesis \textbf{A.1} is employed to develop the DNA-Inverse model (Section \ref{sec:dna_inverse}). However, to ensure uniqueness in the detected timing profile, an additional hypothesis regarding the function $\psi:$ is necessary:

\begin{tcolorbox}[colback=blue!5!white,colframe=blue!75!black]
$\mathbf{\textbf{A.2}}:~ \exists~ \tau_{0} > 0$ such that $\psi_{0} := \psi|_{[0,\tau_{0}]}$ is convex and $\psi_{1} := \psi|_{[\tau_{0},\infty)}$ is concave (or the opposite), and the convexity or concavity of $\psi_{0}$ or $\psi_{1}$ is strict. In addition,  both $\psi_{0},$ $\psi_{1}$ are injective.  
\end{tcolorbox}

These properties will guide the study of problem \eqref{eq:Problem1}. The existence of at most two possible inverses for $\psi$ implies that there will be at most $2^{n}$ inverses for $\Psi$. In this work, we aim to narrow down the possibilities of inverses by investigating other important characteristics of the problem.

\begin{figure}
\centering
\includegraphics[width = 0.95\linewidth]{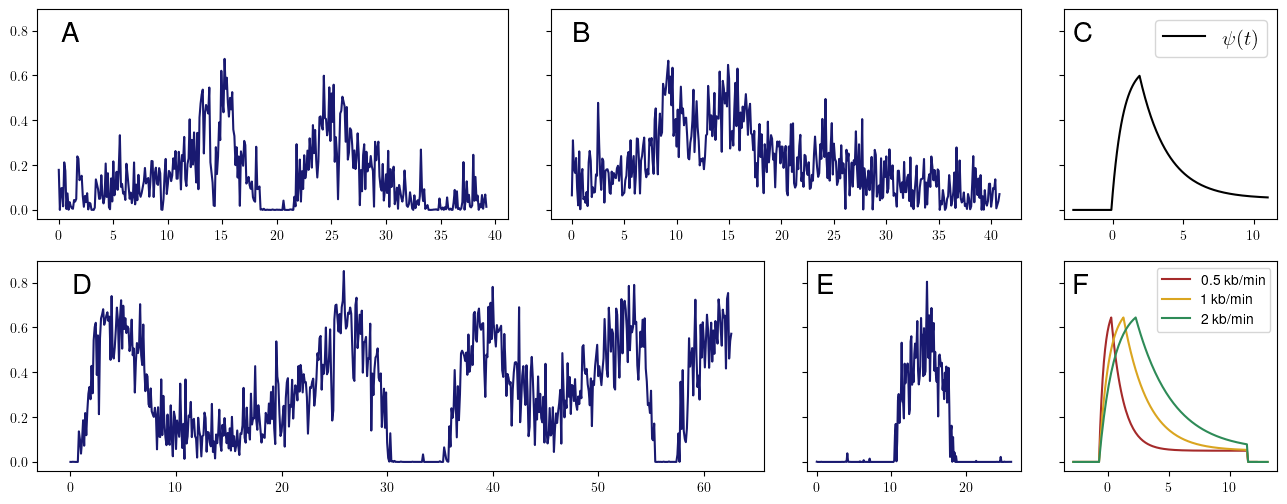}
\caption{(A,B,D,E): examples of experimental yeast FORK-seq signals for different replication configurations. C: function $\psi(t)$ which has been experimentally determined in yeast cells \cite[Section Methods]{biology}. The BrdU is injected in the medium at time $t=0$ , leading to a progressive increase of the intra-cellular concentration of BrdU, called \textit{pulse phase}. After a short period (2~min) the chemical is removed from the medium by dilution resulting in a decrease of the intra-cellular concentration until a certain residual BrdU level, called \textit{chase phase}. For signals in A and D the replication starts after the start of the BrdU injection, resulting in a pattern that reproduces the function $\psi$ with a spatial dilatation depending on the local fork speed (F). Signals in B and E capture an initiation than happened before the injection start (truncated diverging fork atoms) and a termination during the pulse phase (truncated converging fork atoms), respectively. These two configurations can not be recognized by the dictionary model \cite{LageHybrid}}\label{fig:reads}
\end{figure}

\section{Injectivity of $\Psi$}\label{sec:optm_guarantees}
In this section, we establish the uniqueness of the solution to problem \eqref{eq:Problem1} by introducing additional constraints in the available set. In other words, we prove the injectivity of $\Psi$ in the set of piecewise linear signals with extra constraints specified in this section.  These additional constraints are divided in two types: one for analyzing the simple case of linear signals, as detailed in Section \ref{section:sub:tau_linear}, and the second one in the general case of piecewise linear signals, addressed in Section \ref{section:sub:general}. The constraints will be further justified in the context of the DNA replication problem, Section \ref{section:sub:dna}.  This result allows us to demonstrate that the proposed model can uniquely define a \textit{timing profile} $\widehat{\x}$ that better fits the DNA replication signal $\z$. 

We begin by presenting a Lemma that facilitates the manipulation of the set $\mathcal{P}_{C}$ within a continuous framework. 

\begin{lemma}[Continuous form of the set $\mathcal{P}_{C}$]\label{lemma:continuous} Let $\x \in \mathcal{P}_{C}$ and  $ \mathcal{X} := \{x_{1},...,x_{n}\} \subset \mathbb{R},$ be a set such that $x_{i+1} - x_{i} = \Delta x > 0,$ for all $i \in \{1,...,n-1\}.$ Define $f_{\tau}:[x_{1},x_{n}] \rightarrow \mathbb{R}_{+},$ by $f_{\tau}(x_{i}) := \tau_{i},$ for all $i \in \{1,...,n\},$ and let $f_{\tau}$ be linear in $I_{i} := [x_{i},x_{i+1}],$ for $i \in \{1,...,n-1\}.$ Then, $f_{\tau}$ is a continuous piecewise linear function on $[x_{1},x_{n}]$ with $p \leq C$ breakpoints in $(x_{1},x_{n})$. In addition, the breakpoints of $f_{\tau}$ form a subset $\{x_{i_{1}},...,x_{i_{p}}\} \subset \mathcal{X},$ and the indices $\{i_{1},...,i_{p}\} \subset \{1,...,n-1\}$ do not depend on $\mathcal{X}.$ 
\end{lemma}
\begin{proof} Clearly $f_{\tau}$ is continuous and piecewise linear. We will verify that if $\x \in \mathcal{P}_{C},$ $f_{\tau}$ has $p = \|\mathrm{L}\x\|_{0}$ breakpoints. First, we note that there is no breakpoint in the interior of $I_{i},$ for all $i \in \{1,...,n-1\}$ because $f_{\tau}$ is linear by definition. Then, all breakpoints are contained in the set $\{x_{2},...,x_{n-1}\}.$ For $i \in \{2,...,n-1\},$ denote $p_{i} := (x_{i},\tau_{i}) \in \mathbb{R}^{2}.$ The proof is based on the following observation:
$$(\mathrm{L}\x)_{i} = 0 \Leftrightarrow p_{i-1},p_{i},p_{i+1} \text{ are colinear } \Leftrightarrow x_{i} \text{ is not a breakpoint of } f_{\tau}.$$
To see the implication $(\Rightarrow),$ note that $(\mathrm{L}\x)_{i} = 0$ means, by definition: \mbox{$\tau_{i-1} - 2\tau_{i} + \tau_{i+1} = 0,$} which implies 
\begin{equation}\label{eq:taueq}
\mbox{$\tau_{i-1} - \tau_{i} = \tau_{i} - \tau_{i+1}$}.
\end{equation}
Then define $m := \frac{\tau_{i} - \tau_{i-1}}{\Delta x},$ and $c := \tau_{i-1} - m x_{i-1}.$ It is easy to see that $p_{i-1},p_{i},p_{i+1}$ are in the graph of $y(x) = mx + c,$ and then are colinear. We conclude that $x_{i}$ can not be a breakpoint. To see $(\Leftarrow),$ let $i \in \{2,...,n-1\}$ be such that $p_{i-1},p_{i},p_{i+1}$ are colinear. Therefore, there exists $y(x) = mx + c$ such that $p_{i-1},p_{i},p_{i+1}$ are in the graph of $y.$ In this case, it is then easy to verify that \eqref{eq:taueq} holds, which means that $(\mathrm{L}\x)_{i} = 0.$ 

Thus, the set of breakpoints of $f_{\tau}$ writes $\{i_{1},...,i_{p}\} \subset \{1,...,n\},$ and $(\mathrm{L}\x)_{i_{j}} \neq 0,$ for $j \in \{1,...,p\}.$ This set does not depend on the choice of $\mathcal{X}.$ Clearly, $p = \|\mathrm{L}\x\|_{0} \leq C.$
\end{proof}

Equipped with Lemma \ref{lemma:continuous}, we can transition between the vector $\x$ and its continuous counterpart $f_{\x}.$ This lemma will be important in the proof of the main results in this section. 

To investigate the injectivity of $\Psi,$ we consider Assumption \textbf{(A.2)} and two observations: First, since $\psi_{[-\infty,0)} = 0,$ timing profiles cannot be differentiated for negative values $\tau_{i},$ i.e., before the beginning of the experiment. To obtain injectivity we must consider non-negative \textit{timing profiles} $\x.$ We also observe that $\psi$ is not injective. Particularly, there exists $t \in [0,\tau_{0})$ and  $t' \in (\tau_{0},\infty),$ such that $u := \psi(t) = \psi_{0}(t) = \psi_{1}(t') = \psi(t')$. Consequently, the constant vector $\z = \boldsymbol{u} \in \mathbb{R}^{n}$ has two different optimal constant solutions: $\tau = \boldsymbol{t} \in \mathbb{R}^{n},$ and $\tau = \boldsymbol{t'} \in \mathbb{R}^{n}$. For this reason, to obtain a unique solution, we need to restrict the solution set to non-constant vectors:
$$\mathcal{P}_{C}^{\neq} = \left\{ \boldsymbol{\tau} \in \mathcal{P}_{C} : \x \geq 0, \tau_{i} \neq \tau_{i+1} \text{ for all } i \in \{1,...,n-1\} \right\}  \subset \mathbb{R}_{+}^{n}.$$
In the next section, we show that if $\boldsymbol{\tau} \neq \boldsymbol{\tau}'$ are non-constant linear vectors in $\mathcal{P}_{0}^{\neq}$, then  $\Psi(\boldsymbol{\tau}) \neq \Psi(\boldsymbol{\tau}').$

\subsection{Injectivity of $\Psi$ when $\x$ is linear}\label{section:sub:tau_linear}

 We begin by analyzing the case $C = 0$, where the feasible set of \eqref{eq:Problem1} is the set of non-negative linear vectors.  Later, we will extend this argument to non constant piecewise linear vectors.
 

\begin{lemma}[Injectivity of $\Psi$ for $C = 0$]\label{lemma:injectivity} Assume \textbf{(A.2)}. Then, if $n \geq 6,$ the function $\B:\mathcal{P}_{0}^{\neq} \rightarrow \mathbb{R}^{n}$ is injective.
\end{lemma}
\begin{proof}
Consider $\x,\x' \in \mathcal{P}_{0}^{\neq}$ such that $\Psi(\x) = \Psi(\x').$ We will prove that $\x=\x'$. 

Since $\x,\x'$ are linear and non-constant, by lemma \ref{lemma:continuous}, there exists $m,m',c,c' \in \mathbb{R}$, with $m,m' \neq 0$,  such that 
$\tau_{i} = mx_{i} + c,$  $\tau'_{i} = m'x_{i} + c'$ for each $x_{i} \in \mathcal{X},$ for $i \in \mathcal{I} := \{1,...,n\}$.
We begin by partitioning the set $\mathcal{I}$ into three disjoint parts:
$$\mathcal{I}_{+} = \left\{ i \in \mathcal{I} : \tau_{i}, \tau'_{i} \geq  \tau_{0} \}, ~~ \mathcal{I}_{-} = \{i \in \mathcal{I} : \boldsymbol{\tau}, \boldsymbol{\tau}' \leq \tau_{0} \right\}, ~~ \mathcal{I}_{+-} = \mathcal{I}\backslash \left( \mathcal{I}_{+} \cup \mathcal{I}_{-} \right).$$
Observe that for $ i \in \mathcal{I}_{+}$ we have $\Psi(\x)_{i} = \psi_{1}(\tau_{i})$ and similarly with $\x'$. Since 
 $\Psi(\tau) = \Psi(\tau')$ and $\psi_{1}$ is injective because of
 \textbf{(A.2)}, we deduce that 
$$\tau_{i} = \tau'_{i}, ~~ \forall i \in \mathcal{I}_{+}.$$
The same result holds for $i \in \mathcal{I}_{-},$ by the injectivity of $\psi_{0}.$
Since both $\x$ and $\x'$ are linear vectors, they are equal as soon as they coincide at two points, hence if $\# (\mathcal{I}_{+} \cup \mathcal{I}_{-}) \geq 2$ then $\tau = \tau'.$ To complete the proof, it is thus enough to verify that $\# \mathcal{I}_{+-} \leq 4$.
To establish this fact, we proceed by contradiction, and observe first that 
for each $i \in \mathcal{I}_{+-}$ we have
$$\tau_{i} < \tau_{0} < \tau'_{i}, \ \text{or} \ \tau'_{i} < \tau_{0} < \tau_{i}.$$ 
Suppose that $\mathcal{I}_{+-}$ has $5$ elements. By the pigeonhole 
 principle, at least one of the two above inequalities must be satisfied by at least $3$  elements $i \in \mathcal{I}_{+-}$ as illustrated on Figure \ref{fig:lemma}.
Without loss of generality (up to interchanging the role of $\x$ and $\x'$) we can assume that this inequality writes $\tau_{i} < \tau_{0} < \tau'_{i}$, and holds for $i \in \{i_{1},i_{2},i_{3}\} \subset \mathcal{I}_{+-}$ such that $i_{1} < i_{2} < i_{3},$ implying 
$x_{i_{1}} < x_{i_{2}} < x_{i_{3}}$. Consider $t \in (0,1)$ be such that: $x_{i_{2}} = tx_{i_{1}} + (1-t)x_{i_{3}}$  Consider the notation of Assumption \textbf{(A.2)}, where $\psi_{0}$ is strictly concave and $\psi_{1}$ is convex (the other possibilities where $\psi_{0}$ and $\psi_{1}$ exibit different combinations of convexity and concavity can be analyzed similarly). Since $\psi_{0}$ is strictly concave and $\psi_{1}$ is convex, the functions and $m,m' \neq 0$: $\varphi_{0},\varphi_{1}: [x_{i_{1}},x_{i_{3}}] \rightarrow \mathbb{R}$ defined as:
$$\varphi_{0}(x) = \psi_{0}(mx + c), \text{ and } \varphi_{1}(x) = \psi_{1}(m'x + c'),$$ 
are strictly concave and convex respectively.
Since $\Psi(\x) = \Psi(\x')$ the functions $\varphi_{0}$ and $\varphi_{1}$ coincide in three points $x_{i_{1}},x_{i_{2}}$ and $x_{i_{3}}$.
As $\varphi_{0}$ is strictly concave, we get:
$$\varphi_{1}(x_{i_{2}}) =
\varphi_{0}(x_{i_{2}}) < t \varphi_{0}(x_{i_{1}}) + (1-t) \varphi_{0}(x_{i_{3}}) = t \varphi_{1}(x_{i_{1}}) + (1-t) \varphi_{1}(x_{i_{3}}),$$
which contradicts the convexity of $\varphi_{1}$. We conclude that $\# \mathcal{I}_{+-} \leq  4$ consequently $\#(\mathcal{I}_{+} \cup \mathcal{I}_{-}) \geq 2$ and $\tau = \tau'.$ 
\end{proof}

\begin{cor}[Extension to the continuous case]\label{cor:injectivity} Let $I \subset \mathbb{R}$ be an interval. Consider two linear and non-constant functions $f_{1},f_{2}:I \rightarrow \mathbb{R}_{+},$  and the compositions $\varphi^{(1)},\varphi^{(2)}:I \rightarrow \mathbb{R},$ $\varphi^{(1)}(x) = \psi(f_{1}(x))$ and $\varphi^{(2)}(x) = \psi(f_{2}(x)).$ Let $\mathcal{X} = \{x_{{1}},x_{{2}},...,x_{{n}}\}$ be a set of equidistant points, and $n \geq 6.$ Suppose that $\varphi^{(1)}(x_{i}) = \varphi^{(2)}(x_{i}),$ for $i \in \{1,...,n\},$ then $f_{1} = f_{2}.$
\end{cor}
\begin{proof} Consider $\x^{1},\x^{2} \in \mathbb{R}^{n}$ defined by: $\tau^{1}_{i} := f_{1}(x_{i})$
and $\tau^{2}_{i} = f_{2}(x_{i}).$ Then, by assumption, $\x^{1} = \x^{2}$ and $\Psi(\x^{1})_{i} = \psi(\tau^{1}_{i}) = \psi(\tau^{2}_{i}) = \Psi(\x^{2})_{i},$ for all $i \in \{1,...,n\},$ implying $\Psi(\x^{1}) = \Psi(\x^{2}).$ Since $f_{1},f_{2}$ are linear non-constant, $\x^{1},\x^{2}$ do not have constant parts, implying that $\x^{1},\x^{2} \in \mathcal{P}^{\neq}_{0}$. By Lemma \ref{lemma:injectivity}, $\x^{1} = \x^{2},$ and since  $f_{1},f_{2}$ coincide in more than two points, they define the same line and $f_{1} = f_{2}.$
\end{proof}

\subsection{The injectivity of $\Psi$ when $\x$ is piecewise linear}\label{section:sub:general}

In the case where $C > 0$, the feasible set of problem \eqref{eq:Problem1} consists of piecewise linear vectors with less than $C$ breaks. In this case, in addition to the constraint that prevents constant vectors in $\mathcal{P}_{0}^{\neq}$, we need to investigate the distance between two consecutive breaks of $\x$. The intuition of this investigation is that if arbitrarily close breaks are allowed, it is possible to oscillate around $\tau_{0}$ obtaining the same image.

\begin{definition}[Vector of breaks $\boldsymbol{i}^{\tau}$]\label{def:b} Let $\x \in \mathbb{R}^{n},$ consider the indexes $\{i_{1},...,i_{p}\}$ for $p \leq C$ defined in Lemma \ref{lemma:continuous}. Then we define the vector of breaks of $\x$ by: 
$\boldsymbol{i}^{\x} := \left(i_{0},i_{1},...,i_{p},i_{p+1}\right),$ where $i_{0} = 1$ and $i_{p+1} = n.$
\end{definition}

\begin{prop}[Injectivity of $\Psi$ for $C>0$]\label{prop:injectivity} Let $\psi$ be as in Assumption (\textbf{A.2}). Consider:
\begin{equation}\label{eq:P_bar}
\mathcal{P}_{C}^{\geq} = \left\{ \boldsymbol{\tau} \in \mathcal{P}^{\neq}_{C} : \boldsymbol{i}^{\boldsymbol{\tau}}_{k+1} - \boldsymbol{i}^{\boldsymbol{\tau}}_{k} \geq 12, \text{ for all } k \in \{1,...,p+1\}, p = |\boldsymbol{i}^{\x}| \right\},\end{equation}
where  $\boldsymbol{i}^{\boldsymbol{\tau}}$ is defined in Definition \ref{def:b}. Then $\Psi: \mathcal{P}_{C}^{\geq} \rightarrow \mathbb{R}^{n}$ is injective. 
\end{prop}
\begin{proof} 
Consider $\x,\x' \in \mathcal{P}_{C}^{\geq},$ and a set of equidistant points $\mathcal{X} = \{x_{1},...,x_{n}\}.$ Then by Lemma \ref{lemma:continuous} there exists piecewise linear functions $f_{\x},f_{\x'}:[x_{1},x_{n}] \rightarrow \mathbb{R}_{+},$  where $f_{\x}(x_{i}) = \tau_{i},$ and $f_{\x'}(x_{i}) = \tau'_{i},$  for all $i \in \{1,...,n\}.$  To demonstrate that $\x = \x',$ it is sufficient to show that $f_{\x} = f_{\x'}$  

By Lemma \ref{lemma:continuous}, breakpoints of $f_{\x}$ and $f_{\x'}$ are given by $\mathcal{X}^{\x} := \{x_{i_{2}},...,x_{i_{p}}\}$ and $\mathcal{X}^{\x'} := \{x_{i'_{2}},...,x_{i'_{p'}}\}$ for $p,p' \leq C.$ Consider the vector $\boldsymbol{b},$ which aggregates and sorts all breakpoints in $\mathcal{X}^{\x} \cup \mathcal{X}^{\x'}.$ We divide the interval $[x_{1},x_{n}]$ into  intervals $\mathcal{B}_{l} = [b_{l},b_{l+1}],$ for $l \in \{1,...,L-1\}.$  Note that $f_{\x}|_{\mathcal{B}_{l}}$ and $f_{\x'}|_{\mathcal{B}_{l}}$ are both linear by the definition of $\boldsymbol{b}$. Consider $\mathcal{I}_{l} := \mathcal{B}_{l} \cap \mathcal{X}.$ Then, by Corollary \ref{cor:injectivity}, for all $l \in \{1,...,L\}$:
\begin{equation}\label{eq:easycase} 
\# \mathcal{I}_{l} \geq 6 \Rightarrow f_{\x}|_{\mathcal{B}_{l}} = f_{\x'}|_{\mathcal{B}_{l}},
\end{equation}
Note that, according to the structure of $\mathcal{P}_{C}^{\geq},$ since $\boldsymbol{i}^{\x},\boldsymbol{i}^{\x'}$ contains the borders $\{1,n\}$, the first and last breakpoints of $f_{\boldsymbol{\tau}}$ and $f_{\boldsymbol{\tau}'}$ can not appear before 12 points, implying that $\#\mathcal{I}_{1},\#\mathcal{I}_{L} > 6.$
We proceed by investigating the case $\# \mathcal{I}_{l} < 6$. Define:
$$\mathcal{L} = \left\{ l \in \{2,...,L-1\} : \#\mathcal{I}_{l} < 6 \right\}.$$
The objective is to show that for all $l \in \mathcal{L},$ $f_{\x}|_{\mathcal{B}_{l}} = f_{\x'}|_{\mathcal{B}_{l}}$. We begin by proving that $l \in \mathcal{L}$ implies $l-1 \notin \mathcal{L}$ and $l+1 \notin \mathcal{L}$ and we proceed by contradiction. By definition of $\mathcal{I}_{l},$ a breakpoint for  either $f_{\x}$ or $f_{\x'}$ occurs in $l$ and in $l+1.$ If $l+1 \in \mathcal{L},$ an additional breakpoint emerges in $l+2.$ Consequently, $f_{\x}$ or $f_{\x'}$ would accumulate two breakpoints in an interval of length less then $12$ which means that $\tau$ or $\tau'$ is not in the set $\mathcal{P}_{C}^{\geq}$. The same argument applies for the adjacent interval $\mathcal{I}_{l-1}.$ We conclude that $l-1$ and $l+1$ are not in $\mathcal{L}.$  To complete the proof we demonstrate that:
$$\text{If } l-1,l+1 \notin \mathcal{L}, \text{ then } f_{\x}|_{\mathcal{B}_{l}} = f_{\x'}|_{\mathcal{B}_{l}}.$$
 Note that adjacent intervals share boundaries, thus for every $l \in \{1,...,L-1\},$ $\boldsymbol{b}_{l+1} \in \mathcal{B}_{l} \cap \mathcal{B}_{l+1}$ and  $\boldsymbol{b}_{l} \in \mathcal{B}_{l} \cap \mathcal{B}_{l-1}.$ If $l-1,l+1 \notin \mathcal{L},$ applying \eqref{eq:easycase}, we have that: $f_{\x}|_{\{\boldsymbol{b}_{l},\boldsymbol{b}_{l+1}\}} = f_{\x'}|_{\{\boldsymbol{b}_{l},\boldsymbol{b}_{l+1}\}}.$ Since $f_{\x}|_{\mathcal{B}_{l}}$ and $f_{\x'}|_{\mathcal{B}_{l}}$ are linear and coincide in two points, we conclude that $f_{\x}|_{\mathcal{B}_{l}} = f_{\x'}|_{\mathcal{B}_{l}}.$

\end{proof}

\begin{figure}[t]
\centering
\includegraphics[width = 1\linewidth]{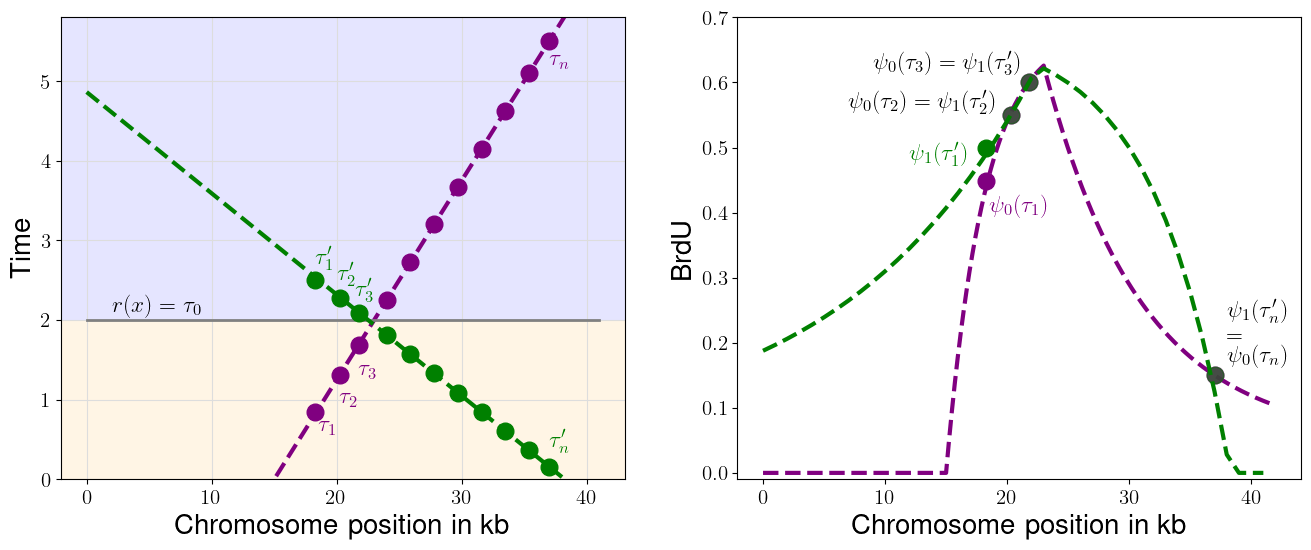}
\caption{Illustration of the injectivity of $\Psi$ for $C = 0$. (left) Purple dots represent the vector $\x = (\tau_{1},\tau_{2},\tau_{3},...,\tau_{n})$ and green dots $\x' = (\tau'_{1},\tau'_{2},\tau'_{3},....,\tau'_{n})$. Note that $\x,\x' \in \mathcal{P}_{0}.$ The dashed lines represent the linear functions that support vectors $\x,\x'$ as stated in Lemma \ref{lemma:continuous}. We note that $\psi$ is injective for vector components in the blue and orange regions. The main challenge in verifying $\Psi$ injectivity is then to consider vectors $\x,\x'$ with components in different sides of the line $r(x) = \tau_{0},$ as those in the figure. (right) Dashed lines represent the images by $\Psi$ of the lines that support vectors $\x$ (purple) and $\x'$ (green). The images  coincide at only three coordinates: ${1},{2}$ and ${n}$. However, $\x$ and $\x'$ have different images. For example, in the case of component ${3}:$ $\Psi(\x)_{3} = \psi_{0}(\tau_{3}) \neq  \psi_{1}(\tau_{3}) = \Psi(\x')_{3}$. This case illustrates the proof of Lemma \ref{lemma:injectivity}. In this lemma, we show that the set $\mathcal{I}_{+-},$ that consists on indexes for which components of  $\tau,\tau'$ are in different sides of the line $r(x) = \tau_{0},$ can not have three elements with coincident images.
}\label{fig:lemma}
\end{figure}

\subsection{Constraints in the set $\mathcal{P}_{C}$ and the DNA replication context}\label{section:sub:dna} 
In this section, we discuss how the constraints of the set $\mathcal{P}_{C}$ of Proposition \ref{prop:injectivity} can be interpreted from the perspective of their application to the DNA analysis problem. In the context of DNA replication analysis, constant parts in vector $\boldsymbol{\tau}$  means that an interval of the DNA fragment $\mathcal{X}$ was replicated simultaneously, which is not in the physical hypothesis of the real problem. On the other hand, the spacing between break-points in the context of the application corresponds to the distance between two initiation to termination events. If these two events are very close, it is not possible for biologists to extract information from the BrdU signal, which makes this hypothesis reasonable. The spacing suggested by Proposition \ref{prop:injectivity} is 12 space units, which corresponds to 1.2 kb in the actual signal. The non negativeness proposed in \eqref{eq:PC} also arises from the applied intuition to the problem. 

\section{DNA-inverse optimization}\label{sec:dna_inverse}

 Existing numerical methods for nonlinear inverse problems  provide local solutions to problem \eqref{eq:Problem1} \cite{Valkonen_accelerationprimaldual,ADMM_convergence_nonconvex}. However this appraoch faces two major limitations: (i) They are not able to provide a global solution, which is a major challenge in the DNA-inverse model; and (ii) They are not fast enough to allow the exploration of different starting points. In this section, we reformulate problem \eqref{eq:Problem1} in order to provide a numerical method able to achieve global solutions .

Employing the notation of Assumption $\textbf{(A.2)},$ the inverse set $\psi^{-1}(b),$ for any $b \in \mathbb{R}_{+},$ can be written as:
\begin{equation}\label{eq:defBinv}
\psi^{-1}(b)  = \psi_{0}^{-1}(b) \cup \psi_{1}^{-1}(b), 
\end{equation}
where $\psi_{0}^{-1}(b)$ and $\psi_{1}^{-1}(b)$  are inverse sets of $\psi_{0}$ and $\psi_{1}$ respectively. Because of $(\textbf{A.2}),$ these sets consist in single elements or are empty. Consider a signal $\z \in \mathbb{R}^{n}.$ Each $\db \in \{0,1\}^{n}$ is associated to a part of the inverse set $\Psi^{-1}(\z),$ given by:
\begin{equation}\label{eq:Bdinv}
 \Psi^{-1}(\z; \db) := \psi^{-1}_{d_{0}}(\mathrm{z}_{0}) \times ...\times \psi^{-1}_{d_{n}}(\mathrm{z}_{n}) \subset \mathbb{R}^{n},   
\end{equation}
where some of the inverse sets on the right hand side can be empty. The set $\mathcal{K}(\z) \subset \{0,1\}^{n}$ select elements for which all inverse sets are composed by a single element:

\begin{equation}\label{eq:def:K}
\mathcal{K}(\z) = \left\{ \db  : \psi_{d_{i}}^{-1}(\mathrm{z}_{i}) \neq \varnothing  \right\} \subset \{0,1\}^{n}.
\end{equation}
For elements $\db \in \mathcal{K},$ $\Psi^{-1}(\z;\db)$ can be identified as a vector in $\mathbb{R}^{n}.$ Using this abuse of notation, we denote $\Psi^{-1}_{\db}(\z)$ a vector in $\mathbb{R}^{n}$ when $\db \in \mathcal{K}:$ 
\begin{equation}\label{eq:Psiinverse_in_K}
\db = (d_{0},...,d_{n}) \in \mathcal{K}(\mathrm{z}) \Rightarrow \Psi^{-1}_{\db}(\z) := (\psi^{-1}_{d_{1}}(\mathrm{z}_{1}),...,\psi^{-1}_{d_{n}}(\mathrm{z}_{n})) \in \mathbb{R}^{n}.
\end{equation}
In this case, it is possible to develop a \textit{Taylor expansion} of the function $\psi,$ which is also a coordinatewise expansion of the function $\Psi.$*

\paragraph{Taylor expansion} Consider a signal $\z \in \mathbb{R}^{n},$ and any coordinate $i \in \{1,...,n\}.$ Let \mbox{$\db \in \mathcal{K}(\z) \subset \{0,1\}^{n},$} and $\x \in \mathbb{R}^{n}_{+}.$ Note that for all $i \in \{1,...,n\}:$ $z_{i} = \psi_{d_{i}}(\psi_{d_{i}}^{-1}(z_{i})) = \psi(\psi_{d_{i}}^{-1}(z_{i})).$ Then, if $\psi$ is continuously differentiable, by Taylor's theorem applied in the point $\psi_{d_{i}}^{-1}(\z_{i}),$ for all $i \in \{1,...,n\},$ we have:
\begin{equation}\label{eq:taylor}
  \mathrm{z}_{i}  -  \psi(\tau_{i}) = \psi^{'}(\psi^{-1}_{d_{i}}(\mathrm{z}_{i}))\left( \psi^{-1}_{d_{i}}(\mathrm{z}_{i}) - \tau_{i} \right) + h(\tau_{i})\left( \psi^{-1}_{d_{i}}(\mathrm{z}_{i}) - \tau_{i} \right)
\end{equation}
where $\lim_{\tau_{i} \mapsto \psi^{-1}_{d_{i}}(\mathrm{z}_{i})} h(\tau_{i}) = 0.$
Define: $w_{\db,i} := \psi^{'}(\psi^{-1}_{d_{i}}(\mathrm{z}_{i})) \in \mathbb{R},$ for all $i \in \{1,...,n\}.$ Then, if $|\psi^{-1}_{d_{i}}(\mathrm{z}_{i}) - \tau_{i}|$ is small for all $i \in \{1,...,n\}:$
\begin{equation}\label{eq:norm}
    \|\z - \Psi(\x)\|^{2}_{2} \approx \displaystyle\sum_{i=1}^{n} w^{2}_{\db,i}\left(\Psi_{\db}^{-1}(\z) - \x \right)_{i}^{2}.
\end{equation}   
The right hand side of \eqref{eq:norm} is a distance between $\Psi^{-1}_{\db}(\z)$ and $\tau,$ that depends on $\db.$ The variable $\db$ adjusts this distance to  the local behavior of $\Psi.$  To formalize this intuition, we define a pseudo-norm that coincide with this notion of distance for $\db \in \mathcal{K}(\z).$

\begin{definition}[weighted norm] Given a vector $\w \in \mathbb{R}^{n}$ and any $\boldsymbol{v} \in \mathbb{R}^{n},$ the weighted norm $\|\boldsymbol{v}\|_{\w}$ with respect to the vector $\w$ is defined as:
$$\|\boldsymbol{v}\|_{\w} := \sqrt{\sum_{i=1}^{n}w_{i}^{2}v_{i}^{2}}.$$
\end{definition}

\begin{definition}[$\|\cdot\|_{\w_{\db}}$]\label{def:defw} Let $\z \in \mathbb{R}^{n}_{+},$ and $\db \in \{0,1\}^{n}.$ For any $\boldsymbol{v} \in \mathbb{R}^{n},$ we define: $\| \boldsymbol{v} \|_{\w_{\db}},$ the weighted norm with respect to the vector $\w_{\db},$
where:
\begin{equation}\label{eq:defw}
w_{\db,i} := d_{i} \odot w_{\0,i} + (1 - d_{i})\odot w_{\1,i}, \text{ for all } i \in \{1,...,n\},
\end{equation}
and
\begin{minipage}[t]{0.48\linewidth}
\begin{equation*}
 w_{\0,i} :=    
\begin{cases} 
 \psi^{'}(\psi_{0}^{-1}(\mathrm{z}_{i}))  & \psi^{-1}_{0}(\mathrm{z}_{i}) \neq \varnothing \\
 0 & \psi^{-1}_{0}(\mathrm{z}_{i}) = \varnothing,
 \end{cases}
  \end{equation*}
\end{minipage}%
\hfill%
\begin{minipage}[t]{0.48\linewidth}
\begin{equation*}
 w_{\1,i} :=    
\begin{cases} 
 \psi^{'}(\psi_{1}^{-1}(\mathrm{z}_{i}))  & \psi^{-1}_{1}(\mathrm{z}_{i}) \neq \varnothing \\
 0 & \psi^{-1}_{1}(\mathrm{z}_{i}) = \varnothing.
 \end{cases}
  \end{equation*}
\end{minipage}

\end{definition}
The definition \ref{def:defw} is illustrated in Figure \ref{fig:weights}. Note that when $\db \in \mathcal{K}(\z),$ the definition of $\w_{\db}$ coincides with the weights of the Taylor approximation \eqref{eq:norm}. In Section \ref{section:sub:noisy}, we discuss this definition to other values of $\db \in \{0,1\}^{n}.$

\subsection{Noiseless signal}
The Taylor expansion suggests that Problem \eqref{eq:Problem1}, defined using the Euclidean metric in $\mathbb{R}^{n}$, can be approximated by a problem that employs an alternative notion of distance given by $\|\cdot\|_{\w_{\db}}$.
\paragraph{Alternative optimization problem in the noiseless case:} Consider $\z \in \mathbb{R}_{+}^{n},$ and the following optimization problem: 
\begin{equation}\label{eq:P2noiseless}
\tag{\textbf{P2}}
  (\x^{*},\db^{*}) := \argmin_{\{(\x,\db)  \in  \mathcal{P}_{C} \times \mathcal{K}(\z)\}} \|\x - \Psi_{\db}^{-1}(\z)\|^{2}_{w_{\db}},   
\end{equation}
where $\|\cdot\|_{w_{\db}}$  is defined in Definition \ref{def:defw} and $\Psi^{-1}_{\db}(\z)$ is defined as in \eqref{eq:Psiinverse_in_K}. 

In Figure \ref{fig:diagram_proof}, we observe possible different positions of the groundtruth timing profile $\bar{\x},$ and the  solutions $(\x^{*},d^{*})$  of \eqref{eq:P2noiseless}  and $\widehat{\x}$ of \eqref{eq:Problem1}. 

\begin{prop}[Noiseless case]
Given any $\bar{\x} \in \mathcal{P}_{C}^{\geq}$ defined in \eqref{eq:P_bar}, consider $\z = \Psi(\bar{\x}).$  Then \eqref{eq:Problem1} and \eqref{eq:P2noiseless} both  admit the same unique solution in  $\mathcal{P}_{C}^{\geq},$ which is precisely $\x.$
\end{prop} 
\begin{proof}
Clearly $\widehat{\x} = \bar{\x}$ is a solution for \textbf{(P1)}. On the other hand, consider $\bar{\db} \in \{0,1\}^{n}$ defined as: $\bar{d}_{i} = 0,$ if $\bar{\tau}_{i} \leq \tau_{0}$ and $\bar{d}_{i} = 1,$ if $\bar{\tau}_{i} > \tau_{0},$ for all $i \in \{1,...,n\}.$ Then $\bar{\x} = \Psi_{\bar{\db}}^{-1}(\z),$  and $\db \in \mathcal{K}(\z),$ implying that $(\bar{\x},\bar{\db})$ is a solution for problem \eqref{eq:P2noiseless}. Because of Proposition \ref{prop:injectivity}, if $\bar{\x} \in \mathcal{P}_{C}^{\geq},$ this solution is unique on this set. 
\end{proof}

\begin{figure}[t]
\centering
\includegraphics[width = 0.9\linewidth]{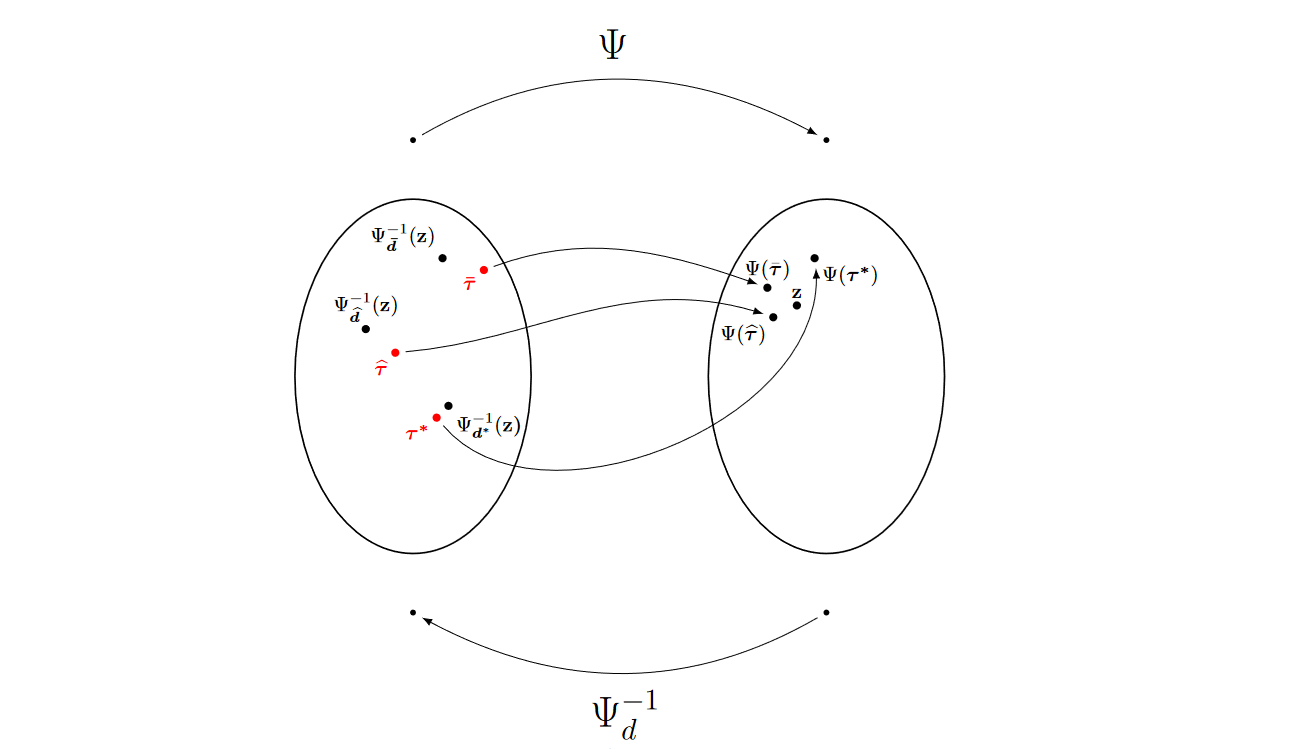}
\caption{Illustration of formulation \eqref{eq:Problem1} and \eqref{eq:P2noiseless}. Elements in red are in the set $\mathcal{P}_{C},$ the set of piecewise linear vectors with at most $C$ breakpoints. On the right, we observe the set where Problem \eqref{eq:Problem1} is formulated, with the Euclidean distance defining the fidelity term. With formulation \eqref{eq:P2}, we transfer the optimization problem to the domain of $\Psi$ using an auxiliary integer variable $\db \in \{0,1\}^{n}$. Depending on $\db$, the inverse $\Psi^{-1}_{\db}(\z)$ is situated in different parts of the domain, justifying the introduction of $\|\cdot\|{w_{\db}}$ to account for the local behavior of $\Psi$  }\label{fig:diagram_proof}
\end{figure}

\subsection{Noisy signal}\label{section:sub:noisy}
In the case of a noisy signal $\z \in \mathbb{R}^{n},$ there is no guarantee that the optimal solution $\db^{*}$ of \eqref{eq:P2noiseless} will be in the set $\mathcal{K}(\z).$ To see that, write $\z = \Psi(\bar{\x}) + \epsilon,$ where $\bar{\x} \in \mathcal{P}_{C}$ and $\epsilon$ is a random vector of dimension $n.$ Let $\bar{\z} = \Psi(\bar{\x}).$ If $\epsilon$ is small enough, we expect the solution of problem \eqref{eq:P2noiseless} to be $(\x^{*},\db^{*}),$
where  $\x^{*} = \bar{\x},$ and $\db^{*}$ is such that $\bar{\x} = \x^{*} = \Psi_{\db^{*}}^{-1}(\bar{\z}).$ Clearly $\db^{*} \in \mathcal{K}(\bar{\z}).$ On the other hand, there is no reason for $\db^{*} \in \mathcal{K}(\z).$
To extend Problem \eqref{eq:P2noiseless} to a noisy signal, we need to define the vector $\Psi_{\db}^{-1}(\z)$ when $\db \notin \mathcal{K}(\z).$  In this extension, we take into account indices $i \in \{1,...,n\}$ for which  $\psi^{-1}(\mathrm{z}_{i}) = \varnothing.$ 

\begin{definition}[$\Psi^{-1}_{\db}(\z)$ and $\|\cdot\|_{\w_{\db}}$]\label{def:Psi_w_noisy} Let $\z \in \mathbb{R}^{n}_{+},$ and $\db \in \{0,1\}^{n}.$ 
We extend the definition of $\Psi_{\db}^{-1}(\z)$ for all $\db\in \{0,1\}^{n}.$ For $i\in \{1,...,n\}$:
\begin{equation}\label{eq:defPsi_inv}
 \Psi^{-1}_{\db}(\z)_{i} :=    
\begin{cases} \psi^{-1}_{d_{i}}(\z_{i}),
  & \text{If } \psi^{-1}_{d_{i}}(\z_{i}) \neq \varnothing\\
  \infty, &  \text{If }  \psi^{-1}_{d_{i}}(\z_{i}) = \varnothing,
 \end{cases}
\end{equation}
where the notation $\psi^{-1}_{d_{i}}(\mathrm{z}_{i})$ is used both for the inverse set and inverse function. Denote $\bar{\mathbb{R}} = \mathbb{R} \cup \{\infty\},$ and $\w_{\db}$ defined in Definition \ref{def:defw}. For any $\boldsymbol{v} \in \bar{\mathbb{R}}^{n},$ we define: 
\begin{equation}\label{eq:defnormw}
\| \boldsymbol{v} \|_{\w_{\db}} = \sqrt{\sum_{\{i \in [n] : \boldsymbol{v}_{i} \neq \infty\}} w_{\db,i}^{2}v_{i}^{2}}.
\end{equation}
\end{definition}
In practice, Definition \ref{def:Psi_w_noisy} extends the definition of $\|\cdot\|_{\w{\db}}$ to values where $\db \notin \mathcal{K}(\z)$, but considers only indices $i$ where $\psi_{d_{i}}^{-1}(z_{i})$ is a singleton. Two reasons can motivate this definition. First, in the noiseless case, weights extend continuously to zero in the region where $\psi$ is not surjective. Observing Figure \ref{fig:linear_piecewise_linear_rep}~C, and defining $a := \lim_{t \rightarrow \infty}\psi_{1},$ we note that $\psi_{1}$ becomes indefinitely close to a constant function as $t$ tends to infinity, causing the corresponding weight $\w_{1,i}$ to tend to zero for indices $i \in \{1,...,n\}$ where $z_{i}$ is close, but above, $a$. In this context, a natural extension to the weights for indices $i\in \{1,...,n\}$ where $z_{i}$ is below $a$ is zero. In the noisy case, $\db \notin \mathcal{K}(\z)$ can also be the effect of noise, for example, in regions where $\z$ is above the maximum $\psi_{\text{max}} := \max_{t \in [0,\infty)}\psi(t)$. In these cases, the choice of not taking into account these indices in $\|\cdot \|_{\w_{\db}}$ is motivated by the numerical results presented in Section \ref{sec:numresults}.

\paragraph{Alternative optimization problem for the noisy case:} The extension of problem \eqref{eq:P2noiseless} for a noisy signal $\z$ reads: 
\begin{equation}\label{eq:P2}
\tag{\textbf{P2}'}
  (\x^{*},\db^{*}) := \argmin_{\{(\x,\db)  \in  \mathcal{P}_{C} \times \{0,1\}^{n}\}} \|\x - \Psi_{\db}^{-1}(\z)\|^{2}_{w_{\db}},   
\end{equation}
where $\|\cdot\|_{w_{\db}}$ and $\Psi^{-1}_{\db}(\z)$ are defined in Definition \ref{def:Psi_w_noisy}.


\section{Numerical approach}\label{sec:numericalmethod}

Note that the \textit{mixed integer nonlinear problem} \eqref{eq:P2} can be reformulated as follows:
\begin{equation}\label{eq:Problem2}
\begin{array}{lll}
& \underset{\x,\db}{\text{min}} & \frac{1}{2}\|\db \odot (\x - \zpulse)\|^{2}_{\w_{\1}} + \frac{1}{2}\|(\boldsymbol{1}-\db)\odot(\x - \zchase)\|_{\w_{\0}}^{2} \\
& ~~\text{s.t.}& \x \in \mathbb{R}^{n},~ \|\mathrm{L}\x\|_{0} \leq C \\
& & \db \in \{0,1\}^{n},
\end{array}
\end{equation}
where the operator $\odot$ represents the coordinatewise multiplication. $\1, \0 \in \{0,1\}^{n}$ are constant vectors, \mbox{$\zpulse := \Psi_{\1}^{-1}(\z)$} and $\zchase := \Psi_{0}^{-1}(\z).$ $\w_{\1}$ and $\w_{\0},$ are defined in \eqref{eq:defw}.

The advantages of this reformulation are twofold: First, for each fixed $\db$, problem \eqref{eq:Problem2} has a quadratic objective function with a non-convex constraint. This optimization problem can be approximately solved by $\ell_{1}$ regularization, in a formulation similar to     \textit{generalized lasso} \cite{generalizedLASSO}. On the other hand, even if the available set in \eqref{eq:Problem2} contains a large set of integer variables, we can attempt to reduce this set based on observations about the nature of solutions of problem \eqref{eq:P2}. This analysis will be developed in Section \ref{sec:constraints}.

\subsection{Combinatorial method for DNA-Inverse}\label{sec:comb_method}
In this section, we present a methodology to address problem \eqref{eq:Problem2}. Our approach involves solving this problem iteratively for each fixed $\db,$ while comparing the obtained optimal values. Clearly, directly applying this method to the entire set $\{0,1\}^{n}$ would be not tractable due to its exponential size. To mitigate this problem, we introduce a subset $\mathcal{D} \subset \{0,1\}^{n}$ in which the optimal solution $\db^* \in \mathcal{D}$. The specific computation of this subset will be discussed in Section \ref{sec:algorithm}.

For each $\db \in \mathcal{D}$, we propose to relax the optimization problem:
\begin{equation}\label{eq:f_criterion}
\begin{aligned}
& \underset{\x}{\text{min}} ~~ \frac{1}{2}\|\db \odot (\x - \zpulse)\|^{2}_{\w_{\1}} + \frac{1}{2}\|(\1-\db) \odot (\x - \zchase)\|_{\w_{\0}}^{2} \\
& ~~\text{s.t.} ~~\x \in \mathbb{R}^{n},~ \|\mathrm{L}\x\|_{0} \leq C \\
\end{aligned}
\end{equation}
into the $l_{1}$ regularized problem:
\begin{equation}\label{eq:piecewise_linear}
\begin{aligned}
\x^{*}_{\db} := &~ \underset{\x}{\text{argmin}} ~~ \frac{1}{2}\|\db \odot (\x - \zpulse)\|^{2}_{\w_{\1}} + \frac{1}{2}\|(\mathrm{I}-\db) \odot (\x - \zchase)\|_{\w_{\0}}^{2} + \lambda \|\mathrm{L}\x\|_{1}  \\
\end{aligned}
\end{equation}
The solutions $\x^{*}_{\db},$ for $\db \in \mathcal{D}$ will be compared using the objective function criterion \eqref{eq:f_criterion}:
\begin{equation}\label{eq:objective_function}
    \x^{*} := \min_{\{\db \in \mathcal{D}\}} F(\x^{*}_{\db}), \text{ where } F(\x^{*}_{\db}) := \frac{1}{2}\|\db \odot (\x_{\db}^{*} - \zpulse)\|^{2}_{\w_{\1}} + \frac{1}{2}\|(\1-\db) \odot (\x_{\db}^{*} - \zchase)\|_{\w_{\0}}^{2}
\end{equation}

\paragraph{Dual of Generalized-lasso:}
Note that problem \eqref{eq:piecewise_linear} can be written as a \textit{generalized-lasso} problem such as in \cite[Equation (2)]{generalizedLASSO}:
\begin{equation}\label{eq:gen_lasso}
 \underset{\x}{\text{argmin}}  \frac{1}{2}\|\w \odot (\x - \zd)\|_{2}^{2} + \lambda\|\Lrm \x\|_{1}  
\end{equation}
where $\z^{\db} = \db \odot \zpulse + (1-\db) \odot \zchase,$ and $\w \in \mathbb{R}^{n}$ is defined by:
\begin{equation}\label{eq:present_w}
    w_{i} = \left\{\begin{array}{ll}
        d_{i}w_{\1,i} + (1-d_{i})w_{\0,i} , & \text{ if } w_{\1,i} \text{ or } w_{\0,i} \neq 0 \\
        0, & \text{ if }  w_{\1,i} = 0 \text{ and } w_{\0,i} = 0.
        \end{array}\right.
\end{equation}
In Theorem \cite[Section 4]{generalizedLASSO}, a quadratic expression for the dual of the generalized-lasso is presented for the case $\w = \1$. We aim to extend this result to the scenario where $\w$ is any vector. To facilitate this extension, we introduce the following notation:
$$\mathcal{I}^{+} = \{i~:~ w_{i} > 0\},~~ \mathcal{I}^{0} = \{i~:~ w_{i} = 0\}$$

\begin{prop} Let $\w^{-1}$ defined by: $w^{-1}_{i} := \frac{1}{w_{i}},$ for all $i \in \mathcal{I}^{+}$ and $w^{-1}_{i} := 0,$ for $i \in \mathcal{I}^{0}.$ Let $\boldsymbol{u}^{*}$ be the solution of the following optimization problem:

\begin{equation}\label{eq:dual}
u^{*} = \begin{array}{lll}
& \underset{\boldsymbol{u}}{\text{argmin}} & \frac{1}{2}\|\w \odot (\Lrm^{\top} \boldsymbol{u})\|^{2} - \langle \mathrm{L}\zd,\boldsymbol{u} \rangle  \\
&\text{s.t.}& \|\boldsymbol{u}\|_{\infty} \leq \lambda \\
& & (\Lrm^{\top}\boldsymbol{u})_{i} = 0,~ \text{for } i \in \mathcal{I}^{0}
\end{array}
\end{equation}

Then, $\x^{*}$ defined by:
\begin{equation}
\left\{ \begin{array}{l}
  \tau_{i}^{*} =  \zd- (\w^2)^{-1} \odot (\Lrm^{\top} \boldsymbol{u}^{*}), ~~ i \in \mathcal{I}^{+} \\
  (\Lrm \x^{*})_{i} = 0 ~~ i \in \mathcal{I}^{0} 
\end{array}\right.
\end{equation}
is a solution of problem \eqref{eq:gen_lasso}.
\end{prop}
\begin{proof} Denote $f(\tau) = \frac{1}{2}\|\w \odot (\x - \z^{d})\|_{2}^{2} + \lambda\|\Lrm \x\|_{1}$ the convex objective function of problem \eqref{eq:gen_lasso}. A solution $\x^{*}$  must satisfy:
$$ 0 \in \partial f(\x^{*}).$$
Note that this equation can be written componentwisely and  when $i \in \mathcal{I}^{0}$, this implies that:
\begin{equation}\label{eq:optm_cond}
(\mathrm{L}\x^{*})_{i} = 0, \text{ for } i \in \mathcal{I}^{0}
\end{equation}
On the other hand, problem \eqref{eq:gen_lasso} can be expressed through its dual problem:
\begin{equation}\label{eq:maxmin}
\begin{aligned}
& \underset{u}{\text{max}}~\underset{\x}{\text{min}} ~~ \frac{1}{2}\|\w \odot (\x - \zd)\|_{2}^{2} + \langle \x , \Lrm^{\top} \boldsymbol{u}\rangle \\
& ~\text{s.t.} ~~\|\boldsymbol{u}\|_{\infty} \leq \lambda. \\
\end{aligned}
\end{equation}
The minimization in $\x$ has the optimal conditions:
\begin{equation}\label{eq:optm_cond2}
\begin{aligned}
 \tau_{i} &= \zd_{i} - ((\w^{2})^{-1} \odot (\Lrm^{\top} \boldsymbol{u}))_{i}, ~~  i \in \mathcal{I}^{+} \\
 \L^{\top}&\boldsymbol{u}_{i} = 0, ~~  i \in \mathcal{I}^{0} 
\end{aligned}
\end{equation}
replacing these conditions in problem \eqref{eq:maxmin}, we obtain the dual variable $u^{*}$ as a solution of problem \eqref{eq:dual}.
Joining conditions in \eqref{eq:optm_cond} and \eqref{eq:optm_cond2} we have:
\begin{equation}
\left\{ \begin{array}{l}
  \tau_{i}^{*} =  \zd- (\w^{2})^{-1} \odot (\Lrm^{\top} \boldsymbol{u}^{*}), ~~ i \in \mathcal{I}^{+} \\
  (\Lrm \x^{*})_{i} = 0 ~~ i \in \mathcal{I}^{0} 
\end{array}\right.
\end{equation}
\end{proof}

\subsection{Constraints of the set $\mathcal{D}$}\label{sec:constraints}
In this section, we discuss how to reduce the set of integer variables in problem \eqref{eq:Problem2}. User
In principle, the solution $\db^{*}$ could be any element of $\{0,1\}^{n}$. A more detailed examination of the problem will show that additional constraints can be incorporated in the available set of \eqref{eq:Problem2} without changing its solution. This analysis will be conducted in two parts: In the first part, we will identify indices for which $\db$ is allowed to oscillate, that is, when $d_{i} = 0$ and $d_{i+1} = 1$ (or the opposite). In the second part, we will consider the values where the signal $\z = 0$ and its impact on the variable $\db \in \{0,1\}^{n}$. We aim defining a subset $\mathcal{D} \subset \{0,1\}^{n}$ of the form:
\begin{equation}\label{eq:Ddef}
    \mathcal{D} = \{\db:  \mathrm{B}\db = \0_{b}, ~\mathrm{A}\db = \0_{a}\},
\end{equation}
where the linear constraints are defined by matrices $\mathrm{A} \in \mathbb{R}^{n,a},$ and  $\mathrm{B} \in \mathbb{R}^{n,b},$ and
$\0_{a} \in \mathbb{R}^{a}$ and $\0_{b} \in \mathbb{R}^{b}$ are constant vectors.

\paragraph{Oscillations on $\db:$} Consider a noiseless signal $\z = \Psi(\bar{\x}),$ such that $\bar{\x} \in \mathcal{P}_{C}.$ Then there exists $\bar{\db} \in \{0,1\}^{n},$ such that $\bar{\x} = \Psi_{\bar{\db}}^{-1}(\z).$ Suppose that $(\bar{\x},\bar{\db})$ is unknown. We aim to define the smallest possible  set $\mathcal{D}$ such that $\bar{\db} \in \mathcal{D}.$ 

For $\bar{\x} \in \mathcal{P}_{C},$ the number of indices for which $\bar{\x}$ crosses $\tau_{0}$ is bounded by $C:$ 
\[\# \{i :  \bar{\tau}_{i} \leq \tau_{0} \leq \bar{\tau}_{i+1} \ \text{ or }  \bar{\tau}_{i+1} \leq \tau_{0} \leq \bar{\tau}_{i} \} \leq C. \]
On the other hand, for $i \in \{1,...,n\}:$ 
\begin{equation}\label{eq:tau}
\bar{\tau}_{i} = \Psi_{\bar{\db}}^{-1}(\z)_{i} = 
\begin{cases}
 \psi_{\0}^{-1}(\mathrm{z}_{i}) \in [0, \tau_{0}], \text{ if } \bar{d}_{i} = 0 \\
 \psi_{\1}^{-1}(\mathrm{z}_{i}) \in [\tau_{0},\infty), \text{ if } \bar{d}_{i} = 1.
\end{cases}
\end{equation}
Then, indices where $\bar{\db}$ change between $0$ and $1$ and indices where $\bar{\x}$ passes through $\tau_{0}$ are the same, as illustrated in Figure \ref{fig:constraints}. For example:  
$$\bar{d}_{i} = 0 \text{ and } \bar{d}_{i+1} = 1 \Rightarrow \bar{\tau}_{i} \leq \tau_{0} \leq \bar{\tau}_{i+1}.$$
Then, we define the set of transitions on $\db$ by:
\begin{equation}\label{eq:I}
\mathcal{I}^{A} := \{ i : \bar{d}_{i} + \bar{d}_{i+1} = 1 \} = \{i : \bar{\tau}_{i} \leq \tau_{0} \leq \bar{\tau}_{i+1} \text{ or } \bar{\tau}_{i+1} \leq \tau_{0} \leq \bar{\tau}_{i}  \}.
\end{equation}
From this characterization, we aim computing $\mathcal{I}^{A}$ using only the known signal $\z$ and the function $\Psi.$  To achieve this, we draw inspiration from the case where the signal $\z$ is defined continuously. Let $z:I \subset \mathbb{R} \rightarrow \mathbb{R}$ be a continuous function, where $I$ is an interval. Consider the functions $\psi^{-1}_{1}(z): I \rightarrow [\tau_{0},\infty),$  $\psi^{-1}_{0}(z): I \rightarrow [0,\tau_{0}]$  and $\bar{d}:I \rightarrow \{0,1\}.$ Let $\bar{\tau}$ be defined as: 
\begin{equation}\label{eq:tau_cont}
\bar{\tau}(x) = 
\begin{cases}
 \psi_{0}^{-1}(z(x)) \in [0, \tau_{0}], \text{ if } \bar{d}(x) = 0 \\
 \psi_{1}^{-1}(z(x)) \in [\tau_{0},\infty), \text{ if } \bar{d}(x) = 1.
\end{cases}
\end{equation}
Denote $f := \psi^{-1}_{1}(z) - \psi_{0}^{-1}(z).$ Then, if $\bar{\tau}$ crosses $\tau_{0},$ it exists $\bar{x} \in I$ such that $\tau(\bar{x}) = \tau_{0},$ implying $f(\bar{x}) = 0.$ Since $f \geq 0,$ $\bar{x}$ is a local minima of $f.$ The reasoning is illustrated in Figure \ref{fig:constraints}. We will employ this continuous intuition in a vectorial context, i.e, detect indices in $\mathcal{I}^{A}$ as local minima of the vector:
\begin{equation}\label{eq:D_localminima}
    \Psi_{\0}^{-1}(\z) - \Psi_{\1}^{-1}(\z).
\end{equation}
note that only the input signal $\z$ is necessary to compute this vector. Then:
\[\mathcal{D} = \{\db \in \{0,1\}^{n}, d_{i} = d_{i+1} \text{ for } i \in \mathcal{I}^{A} \}\]
This strategy is interesting since we can extract \textit{a priori} information about the indices where the solution $\bar{\db}$ can oscillate between $0$ and $1$. Considering that we are able to detect $\mathcal{I}^{A} = \{i_{1},...,i_{J}\}$, we define a matrix $\mathrm{A} \in \mathbb{R}^{n \times a}$ in such a way that for $j \in \{1,...,J\}:$ 
\begin{equation}\label{eq:matrix_A}
\begin{aligned}
\mathrm{A}[i,j] &= 1,  \text{ for } i = i_{j}\\
\mathrm{A}[i,j] &= -1,  \text{ for } i = i_{j} + 1\\
\mathrm{A}[i,j] &= 0,  \text{ for } i \neq i_{j},i_{j} + 1.
\end{aligned}
\end{equation}
If $\db \in \{0,1\}^{n}$ is a solution of problem \eqref{eq:Problem2}, it implies that $\mathrm{A}\db = \0_{m} \in \mathbb{R}^{m}.$

\paragraph{Zeros of  $\z:$ }Another source of information are indices where the signal $\z = 0$. Note that for any $i\in \{1,...,n\}$ $\mathrm{z}_{i} = 0$ implies $\tau_{i} = 0,$ which implies $\db = 0.$ Then the set: $\mathcal{I}^{B} = \{i : \mathrm{z}_{i} = 0\},$ is also important. The set $\mathcal{D}$ is defined by:
\begin{equation}\label{eq:defD2}
\mathcal{D} = \{\db \in \{0,1\}^{n} : d_{i} = d_{i+1} \text{ for } i \in \mathcal{I}^{A},~ d_{i} = 0 \text{ for } i \in \mathcal{I}^{B}\},
\end{equation}
The matrix $\mathrm{B}$ has as columns canonical vectors for indices in  $\mathcal{I}^{B}.$  The set $\mathcal{D}$ can be written as in \eqref{eq:Ddef}. In Section \ref{sec:algorithm}, we discuss how the noise in $\z$ impacts the definition of $\mathcal{D}$ for real data.

\begin{figure}
\includegraphics[width = \linewidth]{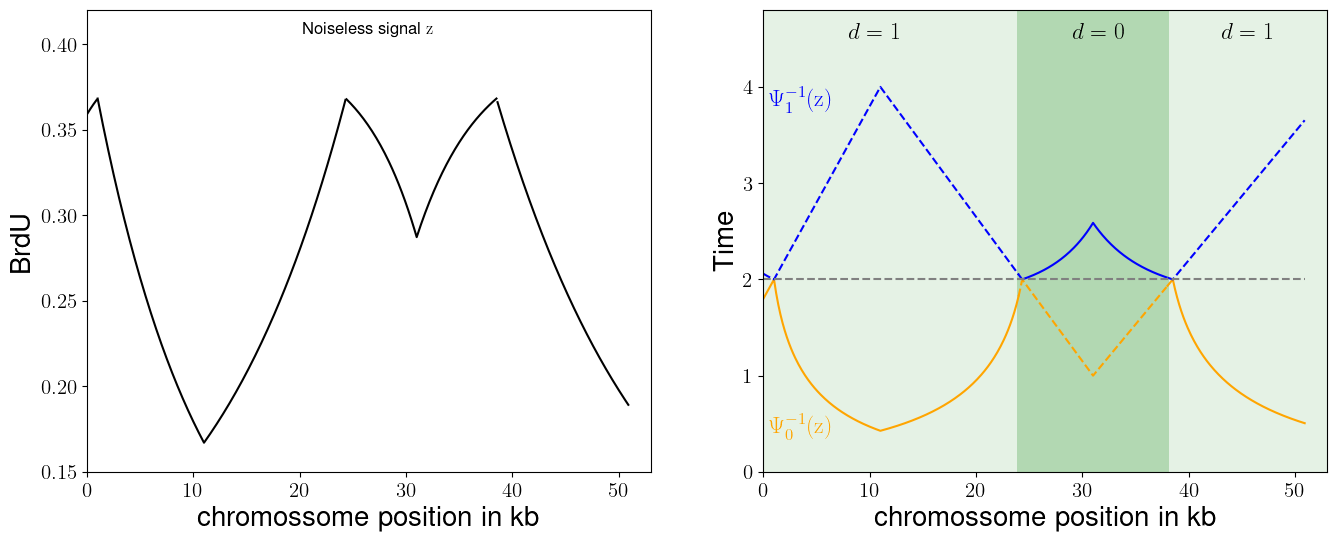}
\caption{Illustration of possible oscillations of $\bar{\db} \in \mathcal{D}.$ (left) Noiseless signal $\z = \Psi(\bar{\x})$ for $\bar{\x} \in \mathcal{P}_{C}.$ (right) Inverses $\Psi_{\1}^{-1}(\z)$ (blue) and $\Psi_{\0}^{-1}(\z)$ (orange). The piecewise linear $\bar{\x}$ can be observed by the dashed lines. For each index $i \in \{1,...,n\},$ when $\bar{\tau}_{i}$ is blue, $\bar{d}_{i} = 1.$ When $\bar{\tau}_{i}$ is orange, $\bar{d}_{i} = 0.$ When $d$ oscillates, $\tau$ crosses the gray dashed line $r(x) = \tau_{0} = 2$. Note the oscillations of $\bar{\db},$ illustrated by transitions in shades of green, coincide with indices where $\Psi_{\1}^{-1}(\z) = \Psi_{\0}^{-1}(\z).$ In addition: $\Psi_{\1}^{-1}(\z) - \Psi_{\0}^{-1}(\z) \geq 0.$ We conclude that oscilations in $\db$ are local minima of the vector $\Psi_{\1}^{-1}(\z) - \Psi_{\0}^{-1}(\z).$ }\label{fig:constraints}
\end{figure}

\subsection{Algorithm}\label{sec:algorithm}
In this section, we discuss the algorithm that provides a numerical solution for Problem \eqref{eq:P2}. The initialization of the algorithm accounts for the high presence of noise in the real DNA replication data,
and occurs in two steps: 1. The computation of the weights $\w_{\db}$ defined \eqref{eq:defw}; 2. The computation of the set $\mathcal{D}$ defined in \eqref{eq:defD2}. The algorithm is described in Algorithm \ref{alg:DNA-Inverse}. 

\paragraph{Weights:} The weights $\w_{\db}$ are computed from weights $w_{\0}$ and $w_{\1},$ for $\0,\1 \in \mathbb{R}^{n}$ in the following way: $\w_{\db} = \db \odot \w_{\1} + (1 - \db) \odot \w_{\0}.$ The calculus of weights $\w_{\0},\w_{\1}$ is done following Definition \ref{def:defw}. For stability reasons, a smoother version $\widetilde{\z}$ of the signal $\z$ replaces the original signal in this calculus. This computation uses the Savitzky-Golay Smoothing Filter that results in $\widetilde{\z}$ such that for each $i\in \{1,...,n\}$: $\widetilde{\mathrm{z}}_{i} = \sum_{i = -2}^{2} c_{n}\mathrm{z}_{i}, $ where $c_{n} = 1/5.$ The weights $\w_{\0},\w_{1}$ are illustrated in Figure \ref{fig:weights} for a noisy signal $\z.$

\begin{figure}[t]
\centering
\includegraphics[width = 0.9\linewidth]{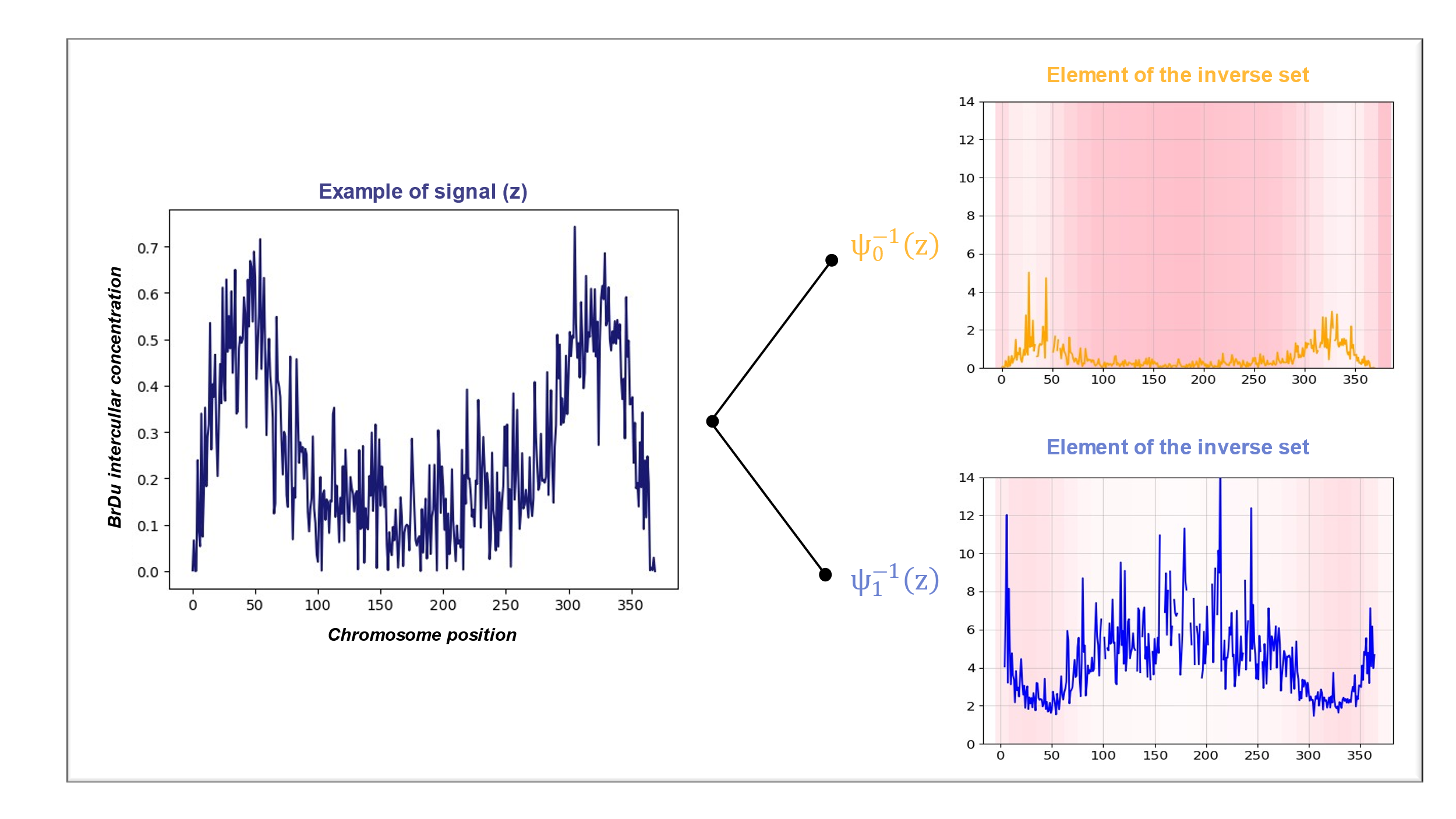}
\caption{Illustration of $\Psi_{\db}^{-1}(\z)$ and $\w_{\db}$ in the noisy case, for $\db = \0,\1 \in \mathbb{R}^{n}$. (left) Noisy read. (right) Vectors $\Psi_{\1}^{-1}(\z)$ and $\Psi_{\0}^{-1}(\z).$ The vector  $\Psi_{\1}^{-1}(\z)$ is not represented in indices where $\Psi_{\1}^{-1}(\z) = \infty.$ We observe that the local behavior of $\Psi$ expands or contracts the noise present in the signal. For this reason, to compare the precision of a piecewise linear approximation of these curves, it is necessary to define $\|\cdot\|_{\w_{\db}}.$ The weights in $\w_{\0},\w_{\1}$ are represented by the intensity of the pink background for $\Psi_{\1}^{-1}(\z)$ and $\Psi_{\0}^{-1}(\z)$ respectively. The more intense the background, the greater the value of $\w_{\0},\w_{\1}$ associated with this part of the signal. We note that more significant noise requires smaller weights.}\label{fig:weights}
\end{figure}

\paragraph{Set  $\mathcal{D}$:} The set $\mathcal{D},$ defined in \eqref{eq:defD2} depends on the set of indices $\mathcal{I}^{B}$ and $\mathcal{I}^{A}.$ For the noiseless case, the indices in $\mathcal{I}^{B}$ are the indices where the signal $\z$ is equal to zero. In the noisy case, this detection stills simple, since the zeros of the signal $\z$ are subjected to considerably less noise then the rest of the signal (see Figure \ref{fig:reads}). Then:
$$\mathcal{I}^{B} = \{i : \mathrm{z}_{i} = 0 \text{ and } \mathrm{z}_{i+1} = 0\}.$$
and $b:= \#\mathcal{I}^{B}.$ In the case of the set $\mathcal{I}^{A},$ as discussed in Section \ref{sec:constraints}, the computation involves the vector: 
\begin{equation}
   \boldsymbol{h} := \Psi_{\0}^{-1}(\z) - \Psi_{\1}^{-1}(\z).
   \label{eq:hdef}
\end{equation}
The indices of local minima of $h$ indicate possible oscillations on $\db$ (from $0$ to $1$ or the opposite). In python, many packages are available to compute this local minima. Since this vector have the same size of the signal $\z,$ which ranges from $100$ to $1000,$ the time required for this computation is negligible.  Due to noise, the local minima of $\boldsymbol{h}$ are treated as the centers of intervals with a certain size: $s^{A}$ (an input parameter) in which $\db$ is allowed to oscillate. Denote $\mathcal{M} = \{i^{\boldsymbol{h}}_{1},...,i_{k}^{\boldsymbol{h}}\}$ the indices of local minimima of $\boldsymbol{h}$. Denote $I_{1}^{\boldsymbol{h}},...,I_{k}^{\boldsymbol{h}},$ the correspondent intervals of size $s^{A}$ centered on the correspondent elements of $\mathcal{M}.$ Then the set $\mathcal{I}^{A}$ is defined as:
$$\mathcal{I}^{A} = \{i \in \{1,...,n\} : i \notin I^{\boldsymbol{h}}_{j} \text{ for } j \in \{1,...,k\}\}.$$
Then $a:= n - ks^{A}.$ Clearly, the set $\mathcal{D}$ is capable of drastically reducing the possible solutions $\db \in \{0,1\}^{n}.$ Nevertheless, the best upper bound for $\mathcal{D}$ cardinality is: $\#\mathcal{D} \leq (s^{A})^{k}$ which stills large. Therefore, we consider a representative subset of $\mathcal{\tilde{D}} \subset \mathcal{D}$. It work as follows: we divide each interval $I_{i}^{\boldsymbol{h}}$ into $m^{A}$ (an input parameter) equal parts, and consider the possible combinations for $i \in \{1,...,k\}$.  Thus, the set $\# \mathcal{\tilde{D}} \leq (m^{A})^{k}$ elements. In general, for DNA replication signals considered in the application, $k$ varies between $2$ and $5$.

\paragraph{Choice of parameters: } The parameter $s^{A}$ is  choose as $60$ units, or $0.6kb.$ The parameter $m$ is chose as $3.$ These choices aim to balance good results with low computation time. Higher values of $m$ directly impact the computation time as they increase the number of elements in $\mathcal{\tilde{D}}$. The value of $s^{A}$ is more associated with the amount of noise and uncertainty regarding the local minima of the vector $h$ in \eqref{eq:hdef}. The parameter $\lambda$ is empirically fixed as $8.$

\paragraph{Algorithm:} According to Section \ref{sec:comb_method}, the DNA-inverse algorithm loops as follows:
 \begin{algorithm}[t]
  \caption{DNA-Inverse}\label{alg:DNA-Inverse}
  \KwData{Input data: $\mathbf{z}.$ Parameters: $s^{A}$ and $m^{A}$.}
 
  \textbf{Initialization:} \\
  Compute weights $w_{\db}$. Compute the sets $\mathcal{\tilde{D}} \subset \mathcal{D}.$ Set: $\mathcal{D}_{\text{past}} = \emptyset$\;
  
  \textbf{Main Loop:} \\
  \For{$\db \in \mathcal{\tilde{D}}$}{

    \textbf{Step 0:} $\mathcal{D}_{\text{past}} \gets \mathcal{D}_{\text{past}} \cup \{\db\}$\;
    \textbf{Step 1:} Solve the optimization problem \eqref{eq:piecewise_linear} by its dual formulation \eqref{eq:dual} (quadratic optimization), obtaining a solution $\x_{\db}^{*}$
    \;
    
    \textbf{Step 2:} Compute the objective of problem \eqref{eq:Problem2} for $\x^{*}_{\db}$ (without $\ell_{1}$ regularization) and compare with objective values of previous 
    $\db \in \mathcal{D}_{\text{past}}:$ 
    \[\db^{*}:= \argmin_{\db \in \mathcal{D}_{\text{past}}} F(\x_{\db}^{*}),~~\x^{*} := \tau_{\db^{*}}^{*},\] 
    $F$ is defined in \eqref{eq:objective_function}.
  }
  
  \textbf{Output:} $\tau^{*},d^{*}$  \\

\end{algorithm}

\section{Numerical results}\label{sec:numresults}

The numerical results presented in this section utilize real data obtained from yeast DNA \cite{biology}. We opt for real data due to the challenge of fully replicating the type of noise present in this dataset. The signals vary in size from $50$ to $1000$, and a total of $300$ reads are analyzed with an average of 2.3 detected forks by read. Biologists, with their trained eye, are capable of recognizing the replication events associated with each of these reads. The objective of this test is to detect these events automatically. The FORQ-seq technique \cite{Forkseq} has enabled scientists to significantly increase the amount of data available for analysis. The challenge now is to develop methods that do not rely on individual analysis of each sample, allowing for a considerable increase in the amount of analyzed data.

The numerical results from the DNA-Inverse method will be divided into two parts: Section \ref{sec:comp} compares DNA-Inverse with an state-of-the-art proximal method capable of providing local minima to  \eqref{eq:Problem1}. In Section \ref{sec:advantages}, we explore the advantages of this method and discuss its relevance for DNA replication analysis.

\subsection{Comparison with state-of-the-art method}\label{sec:comp}

When $\Psi$ is non-linear, the resulting optimization problem \eqref{eq:Problem1} is generally non-convex,  which is a major challenge in optimization.  Recent theoretical and numerical advancements have significantly improved the treatment of problems of type \eqref{eq:Problem1}, showing remarkable flexibility concerning the types of operators $\Psi$ and regularization terms \cite{Valkonen2021,Valkonen_accelerationprimaldual,ADMM_nonconvex,ADMM_convergence_2}.  However, a key limitation of these methods lies in their pursuit of local solutions, due to the non-convex nature of the problem. In relevant applications, such as DNA replication analysis, attaining only local solutions do not provide substantial progress towards achieving the overall objective. In these cases, numerical methods should focus on special applications.

The numerical methods proposed in \cite{Valkonen2021,ADMM_convergence_nonconvex} address problems of type \eqref{eq:Problem1}  considering the $\ell_{1}$ regularization to impose piecewise linear solutions \cite{Valkonen2021,generalizedLASSO}:
\begin{equation}\label{eq:Problem1+l1}
\min_{\x \in \mathbb{R}^{n}}\|\z - \Psi(\x)\|_{2}^{2} + \gamma\|\mathrm{L}\x\|_{1},
\end{equation}
for some $\gamma > 0.$ In this section, we adopt the primal-dual formulation  with the PDPS algorithm  \cite[Algorithm 1]{Valkonen_accelerationprimaldual}. The reason for this choice is that this numerical method is treated specially in the case of non-linear inverse problems such as \eqref{eq:Problem1}, including the choice of parameters involved in iterations. In this context, consider the convex conjugate formula applied to the fidelity term:
\begin{equation}
G(\boldsymbol{u}) = \|\boldsymbol{u} - \z\|_{2}^{2}, ~~~ G^{*}(\boldsymbol{y}) = \sup_{\boldsymbol{u} \in \mathbb{R}^{n}}~\langle \boldsymbol{u},\boldsymbol{y} \rangle - G(\boldsymbol{u}). 
\end{equation}
Replacing the variable $\boldsymbol{u}$ by $\Psi(\x)$, we obtain an equivalent minmax formulation of problem \eqref{eq:Problem1+l1}:
\begin{equation}\label{eq:minmax_form}
\min_{\x \in \mathbb{R}^{n}}\max_{\boldsymbol{y} \in \mathbb{R}^{n}}\gamma\|\mathrm{L}\x\|_{1} + \langle \Psi(\x),\boldsymbol{y} \rangle - G^{*}(\boldsymbol{y}).
\end{equation}
The PDPS algorithm proposes to solve \eqref{eq:minmax_form} using the same principle of proximal point methods. This algorithm iterates over $k \in \mathbb{N}$:
\begin{equation}\label{eq:alg_valk}
\begin{cases}
\x^{k+1} = \text{prox}_{\sigma_{1}\gamma\|\cdot\|_{1}}(\x^{k} - \sigma_{1}\B'(\x^{k})y^{k})\\
\boldsymbol{y}^{k+1} = \text{prox}_{ \sigma_{2}(G^{*} - 2 \langle \Psi(\x^{k}),\cdot \rangle)}(\boldsymbol{y}^{k} - \sigma_{2}\Psi(\x^{k})),
\end{cases}
\end{equation}
for $\sigma_{1},\sigma_{2} > 0.$ In order to ensure the weak convergence of this algorithm \cite[Theorem 1]{Valkonen2021}, $\sigma_{1}$ and $\sigma_{2}$ might respect the following inequality established in \cite[Example 6]{Valkonen2021}:
\begin{equation}
\sigma_{1} \leq \frac{1}{\sigma_{2}L_{\B}^{2} + L_{\B'}\rho_{\boldsymbol{y}}/2}
\end{equation}
where $L_{\B},L_{\B'}$ represent the Lipschitz contants of $\B$ and $\B'$ respectively. And $\rho_{\boldsymbol{y}}$ the radius of the ball where the iterative sequence \eqref{eq:alg_valk} converges. We can easily estimate $L_{\B},L_{\B'} \leq 1$ by computing the derivative of order 1 and 2 of $\B.$ In the following example we compute $\rho_{\boldsymbol{y}} = 1$ by estimating the norm of variable $\boldsymbol{y}$ empirically. We also set $\gamma = 1$ by studying the parameter for which the provided solution $\x$ is piecewise linear. 

\paragraph{Experiment with noiseless signal and local minima: }
Consider a noiseless read: $\z = \Psi(\bar{\x}),$ where
$\z$ and $\bar{\x}$ are illustrated in Figure \ref{fig:simul_DNAInverse}. Problem \eqref{eq:Problem1+l1} is non-convex, implying that a local minimum can not be generalized as a global minimum. For each initial point  $(\x_{\text{initial}},\boldsymbol{y}_{\text{initial}}),$ scheme \eqref{eq:alg_valk} find a local solution of problem \eqref{eq:Problem1+l1}.  In Figure \ref{Fig:comp_initial}, we analyze the dependence of PDPS solution with respect to these initial points. Let $\x_{\text{Initial}}^{i},$ for $ i \in \{1,2,3,4\}$ be possible initial points described as follows:  $\x_{\text{initial}}^{1}$ and $\x_{\text{initial}}^{2}$ are constant vectors with values $0.2$ and $5$ respectively. $\x_{\text{initial}}^{3}$ is generated by an uniform distribution between $0$ and $5,$ and $\x_{\text{initial}}^{4}$ is a perturbation of the ground truth $\bar{\x}.$  In Table \ref{tab:obj_values} we compare the objective value of different local minima and the correspondent execution time. We observe that only $\x^{4}_{\text{initial}}$ is capable to provide the optimal global solution provided by DNA-inverse. In addition, the runtime time of DNA-Inverse is considerably lower then PDPS. In Figure \ref{fig:simul_DNAInverse}, we observe that the DNA-Inverse method closely approximates the original timing profile $\bar{\x}.$ Note that an initial point is not necessary for Algorithm \ref{alg:DNA-Inverse}.

\begin{figure}[t]
\includegraphics[width = 0.95\linewidth]{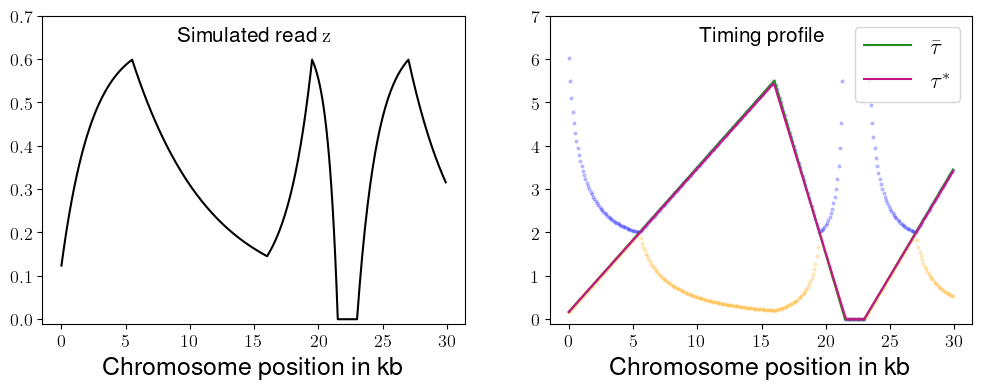}
\captionof{figure}{(left) Simulated read $\mathrm{z} = \Psi(\bar{\x})$. (right) Solution $\x^{*}$ obtained via the DNA-Inverse method (in pink) compared to ground truth $\bar{\x}$ (in green). In orange, $\Psi_{\mathcal{\0}}^{-1}(\z)$, and in blue, $\Psi_{\mathcal{I}}^{\1}(\z)$. We observe a close resemblance between the piecewise linear vectors $\bar{\x}$ and $\x^{*}$.}\label{fig:simul_DNAInverse}
\end{figure}

\begin{table} 
\centering
\begin{tabular}{|c|c|c|c|}
 \hline 
Method &$\x_{\text{initial}}$ & Objective value of local minima & Execution time\\ 
 \hline
 PDPS &$\x_{\text{initial}}^{1}$ &$ 0.18 $& $45$s\\ 
 PDPS & $\x_{\text{initial}}^{2}$ & $0.58$ & $45$s\\ 
 PDPS &  $\x_{\text{initial}}^{3}$ & $0.12$ & $46$s \\ 
 PDPS & $\x_{\text{initial}}^{4}$ & $0.073$ & $40$s \\ 
 DNA-Inverse &  $ -- $ & $0.073$& $7s$ \\ 
 \hline
\end{tabular}
\caption{Comparison between objective value and execution time for different initiation points in the case of a noiseless signal $\mathrm{z}$ illustrated in Figure \ref{fig:simul_DNAInverse}. The convergence criterion for PDPS limits distance between two consecutive iterates: $\|\tau_{k} - \tau_{k+1}\| \leq 1\mathrm{e}{-5}$.}\label{tab:obj_values}
\end{table}

\begin{figure}[t]
\centering
\includegraphics[width = 0.95\linewidth]{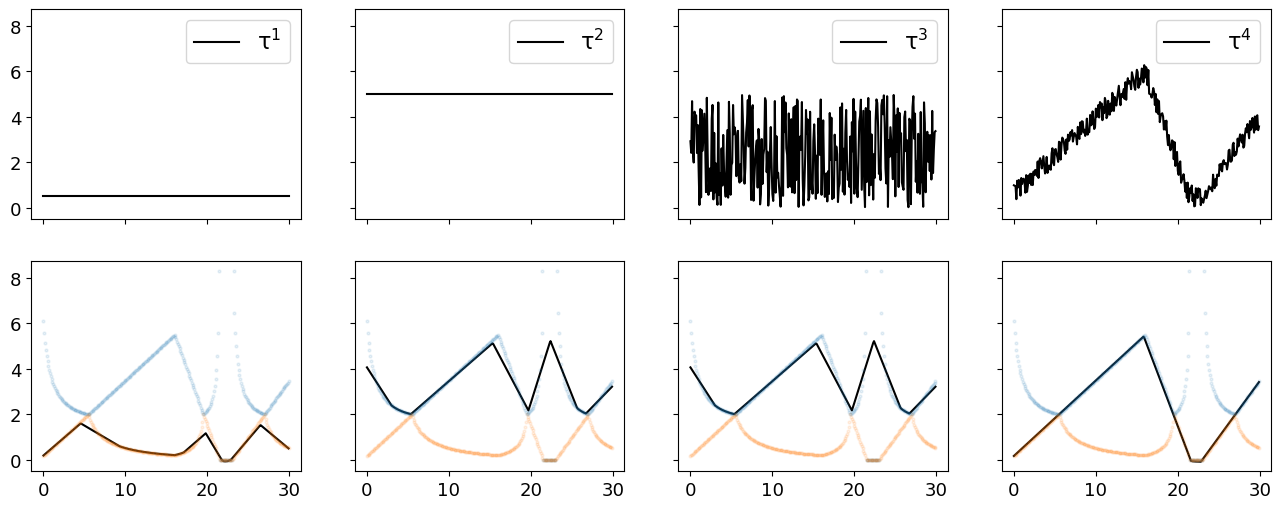}
\caption{Illustration of the results of PDPS algorithm for different initial points. (above): different values of $\x_{\text{initial}}^{i}$ for $i\in\{1,2,3,4\}.$ (below): solutions of the PDPS algorithm with the correspondent initial points. In orange we observe $\z^{\0} = \Psi_{\0}^{-1}(\z)$ and in blue $\z^{\1} = \Psi_{\1}^{-1}(\z).$}\label{Fig:comp_initial}
\end{figure}

\begin{table} 
\centering
\begin{tabular}{ |c|c| c|}
 \hline 
 Method & Mean execution time & Median execution time \\ 
 \hline
 DNA-Inverse & 42s & 10s\\ 
 PDPS & 272s & 163s  \\ 
  \hline
\end{tabular}
\caption{Comparison of execution time between DNA-Inverse and Adapted PDPS }\label{tab:time_comp}
\end{table}

\paragraph{Experiment with different initiation points:} When applying the PDPS method, we find local solutions. These solutions are not arbitrary, they depend on the inverses $\Psi_{\db}^{-1}$ defined in Section \ref{sec:numericalmethod}. More specifically, for an initial point $\x_{\text{initial}},$ consider $d_{\text{initial}}$ defined by:
$$\db_{\text{initial}} = \argmin_{\db \in \{0,1\}^{n}} \|\x_{\text{initial}} - \Psi_{\db}^{-1}(\z)\|_{2}^{2}.$$
Results exposed in Figure \ref{Fig:comp_initial} indicate that the PDPS algorithm results in a piecewise vector that approximates $\Psi_{\db_{\text{initial}}}^{-1}(\z).$ This fact suggests that we can adapt the initialization of DNA-Inverse, Algorithm \ref{alg:DNA-Inverse},  for the PDPS. The adapted algorithm loops as follows: Step 1 - For each $d \in \mathcal{\widetilde{D}},$ we set as initial point the vector $\Psi^{-1}_{d}(\mathbf{z})$; Step 2 -  Compare the objective values for each $d \in \mathcal{\widetilde{D}}$ and chose the smaller one. We call this strategy \textit{Adapted PDPS} and its output: $(\tau_{\text{PDPS}}^{*},d_{\text{PDPS}}^{*}).$  In Figure \ref{fig:comp_valk_dna_inverse} we observe the similarity between solutions for three different elements of the real data-set of yeast reads. For all these cases, the optimal $d_{\text{PDPS}}^{*} = d_{\text{DNA-Inverse}}^{*}$ and $\tau_{\text{PDPS}}^{*} \approx \tau_{\text{DNA-Inverse}}^{*}.$   

As illustrated in Figure \ref{fig:comp_valk_dna_inverse}, the \textit{Adapted PDPS} has shown to be efficient in providing a global solution to problem \eqref{eq:Problem1+l1}. Nevertheless, the main drawback of this adaptation is the execution time. The convergence of the proximal algorithm is slow, resulting in a significantly high total execution time, as displayed in the Figure \ref{Fig:comp_time} and in Table \ref{tab:time_comp}.

\begin{figure}[t]
\centering
\includegraphics[width = 0.95\linewidth]{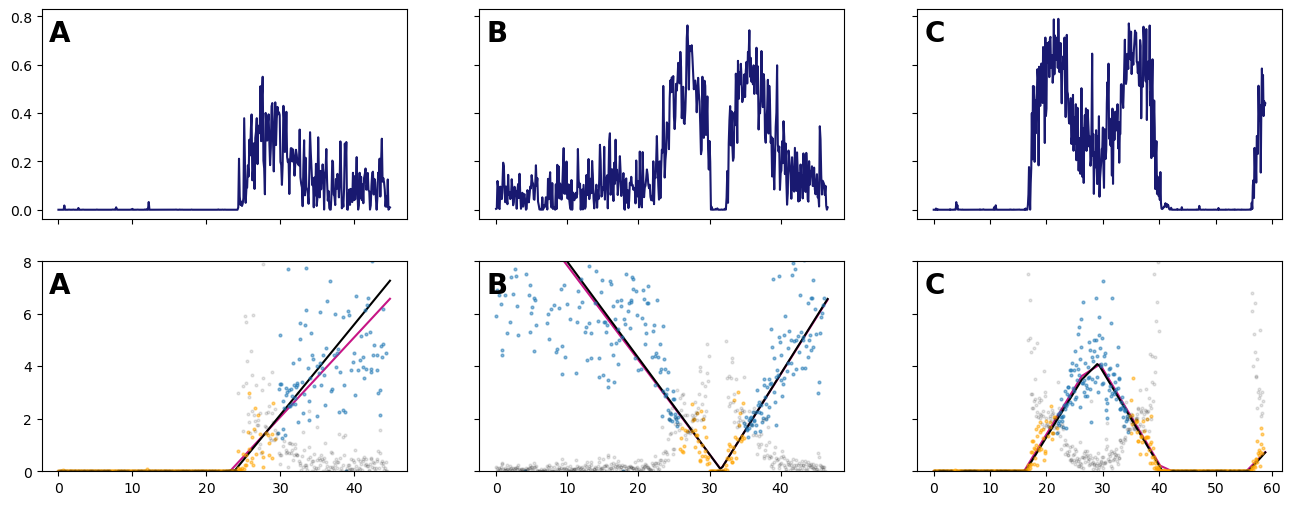}
\caption{(Top) Different reads in blue representing the following events: A. Progression of a fork to the right; B. Initiation; C. Termination. (Bottom) Solution $\x_{\text{DNA-Inverse}}^{*}$ (in pink) and $\x^{*}_{\text{PDPS}}$ (in black) for the different reads. Colored points are selected by the the solution $d^{*}$ that is the same for both methods. For $i \in \{1,...,n\},$ points in blue are the inverse $\psi_{0}^{-1}(\mathrm{z}_{i})$ where $d^{*}_{i} = 0,$ and points in orange are the inverse $\psi_{1}^{-1}(\mathrm{z}_{i})$ when $d^{*}_{i} = 1.$ Black points are those not selected for the solution $d^{*}.$ We emphasize the similarity of the time profiles for PDPS and DNA-Inverse, and the fact that the solutions in variable $\db$ are the same.}\label{fig:comp_valk_dna_inverse}
\end{figure}

\begin{figure}
\centering
\includegraphics[width = 0.6\linewidth]{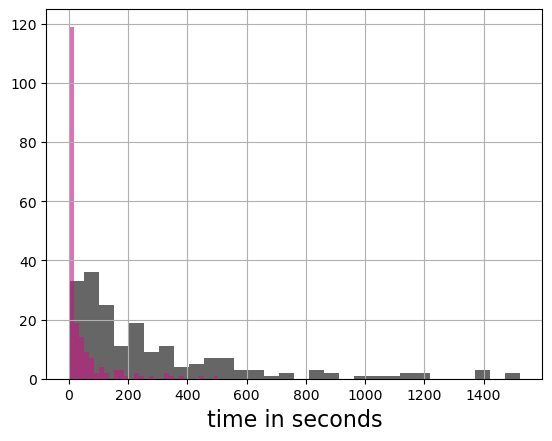}
\caption{Comparison between execution time distribution for DNA-Inverse (pink) and Adapted PDPS (black). The median execution time of DNA-Inverse is 16 times smaller than that of PDPS.}\label{Fig:comp_time}
\end{figure}

\subsection{Advantages of DNA-Inverse with respect to other methods:}\label{sec:advantages}

The primary advantage of the DNA-Inverse algorithm over previous works \cite{biology,LageGretsi,LageHybrid} is its ability to detect any replication event. Unlike previous methods, which only detected replication origins that had been activated before the beginning of the experiment, the DNA-Inverse algorithm is capable of detecting all replication events. In this context, various cases of interest, which are prevalent in the database, were previously excluded from the statistics, as illustrated in Figure \ref{fig:results}.

The effect of $\ell_{1}$ regularization in \eqref{eq:piecewise_linear} is to introduce a bias that tends to yield lower speed values because it affects the angular coefficient of lines \cite{Lasso_bias}. To mitigate this drawback, the DNA-Inverse method enables an enhancement of estimation at low cost. Note that after correctly detecting the variable $\db^*$, the problem  of finding the time profile $\x^{*}$ consists in fitting a piecewise linear function in a noisy data, a task for which various methods can be applied \cite{piecewise_fit_1,piecewise_3}. This possibility is promising for achieving velocity estimation with unprecedented accuracy in future studies.

\begin{figure}
\centering
\includegraphics[width = 0.95\linewidth]{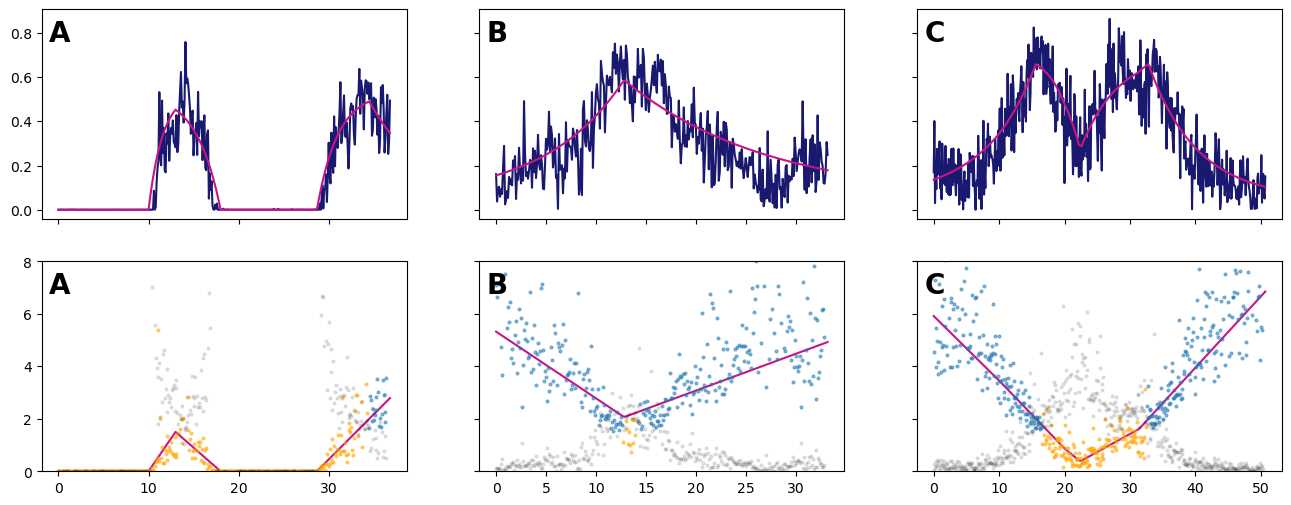}
\caption{(Top) Different reads in blue representing the following events: A. (left) Termination (right) fork progresses to the right; B. Initiation; C. Initiation. The line in pink represent the approximation $\Psi_{d^{*}}(\tau^{*})$ given by the optimal solution of the DNA-Inverse method. (Bottom) Solution $\x^{*}$ for DNA-Inverse (in pink). Colored points are selected by the the solution $d^{*}.$ For $i \in \{1,...,n\},$ points in blue are the inverse $\psi_{0}^{-1}(\mathrm{z}_{i})$ where $d^{*}_{i} = 0,$ and points in orange are the inverse $\psi_{1}^{-1}(\mathrm{z}_{i})$ when $d^{*}_{i} = 1.$ Black points are inverses not selected by the optimal $d^{*}.$  All these events could not be detected in previous works.}\label{fig:results}
\end{figure}

\section{Conclusion}
In this work, we analyzed DNA replication in single-molecule from the perspective of a nonlinear and non-convex inverse problem. We developed the model called DNA-Inverse, which effectively integrates different possibilities of local solutions, eliminating the necessity to solve the problem with multiple initial points. We studied theoretical results demonstrating the coherence of the proposed inverse problem and compare its results with state-of-the-art methods, highlighting its advantage in terms of execution time. Additionally, we showed how the results from this formulation outperform previous works, expanding the amount of data and replication events that can be analyzed by biologists.

\clearpage

\printbibliography

\end{document}